\documentclass[11pt,a4paper]{amsart}
\usepackage[utf8]{inputenc}
\usepackage[T1]{fontenc}
\usepackage{soul}
\usepackage{amsfonts, amsthm, amsmath}
\usepackage{mathrsfs}
\usepackage[shortlabels]{enumitem}
\usepackage{tikz}
\usetikzlibrary{shapes}
\usetikzlibrary{perspective}
\usepackage[labelformat=simple]{subcaption}

\usepackage[colorlinks,linkcolor=black,citecolor=black]{hyperref}

\newtheorem{theorem}{Theorem}[section]
\newtheorem{corollary}[theorem]{Corollary}
\newtheorem{proposition}[theorem]{Proposition}
\newtheorem{lemma}[theorem]{Lemma}

\newtheorem{conjectures}[theorem]{Conjectures}
\theoremstyle{definition}
\newtheorem{example}[theorem]{Example}
\newtheorem{definition}[theorem]{Definition}
\newtheorem{remark}[theorem]{Remark}
\newtheorem{observation}[theorem]{Observation}
\DeclareMathOperator{\tr}{tr}

\DeclareMathOperator{\mult}{mult}
\DeclareMathOperator{\spec}{\sigma}

\DeclareMathOperator{\nullity}{nullity}
\DeclareMathOperator{\Ker}{kernel}

\newcommand{\ns}[1][w]{{\mathcal{R}_{#1}}}

\newcommand{\pendentpath}[1]{pendent $#1$-path}

\newcommand{\cs}{{\mathcal{R}}}
\let\R\cs

\newcommand{\highMultiplicityList}{critical multiplicity list}

\DeclareMathOperator{\h}{ht}

\DeclareMathOperator{\children}{children}

\newcommand{\trans}{^\top}

\newcommand{\bR}{\mathbb{R}}

\newcommand{\bftree}{T_{\mathrm{BF}}}

\newcommand{\PHset}{{\mathcal{PH}}}

\renewcommand{\S}{\mathcal{S}}

\newcommand{\bZ}{\mathbb{Z}}

\newcommand{\ordm}{{\bf m}}
\newcommand{\IEPG}{IEP-$G$}

\newcommand{\RS}{\mathrm{RS}}
\newcommand{\simplexk}{\mathring{\Delta}_{k - 2}}
\newcommand{\simplex}{\mathring{\Delta}}

\newcommand{\seq}{\mspace{-10mu}&=\mspace{7mu}}
\newcommand{\splus}{\mspace{-13mu}&+&\mspace{-13mu}}
\newcommand{\smin}{\mspace{-13mu}&-&\mspace{-13mu}}
\newcommand{\sprox}{\mspace{-12mu}& \approx & \mspace{-12mu}}
\newcommand{\sleft}{\Bigl[&\mspace{-18mu}}
\newcommand{\sright}{\Bigr]}
\newcommand{\om}{\xi}

\newcommand{\ordmrigid}{\ordm_{\mathrm{rig}}}
\newcommand{\Texample}{T_8}

\tikzset{
    every node/.style={draw, circle, fill=white, inner sep=1pt}
    }
    
\newcommand{\inductiveStepPicture}{%
\begin{figure}[t]
  \begin{center}
    \begin{tikzpicture}[scale=0.8]
    \tikzset{every node/.style={}}
      \draw[thick] (5,10)--(6,11)--(7,10);
      \draw[thick] (8,13)--(9,14)--(10,13);
      \draw[thick] (9,14)--(10,15);
      \draw[thick] (7,12)--(6,11);
      \draw[thick] (4,9)--(2,7);
      \draw[thick] (2,7)--(1,6);
      \draw[thick] (3.5,7)--(3,8);
      \draw[thick,dashed,orange] (1,2)--(1,6);
      \node[draw,circle,fill=orange] at (3,8) {};
      \node[draw,circle,fill=orange] at (6,11) {};
      \node[draw,circle,fill=orange] at (9,14) {};
      \foreach  \x in {2,6,7,8}{
      \draw (5,\x) node[anchor=west]{$\cdots$};
      };
      \foreach  \x in {9,10,11}{
      \draw (8,\x) node[anchor=west]{$\cdots$};
      };
      \foreach  \x in {12,13,14}{
      \draw (11,\x) node[anchor=west]{$\cdots$};
      };
      \foreach  \x in {(5,10),(8,13)}{
        \begin{scope}[shift=(\x)]
     	  \node[draw,circle,fill=white] at (0,0) (A) {};
          \node[draw,circle,fill=white] at (-1,-1) (B1) {};
          \node[draw,circle,fill=white] at (-0.5,-1) (B2) {};
          \node[draw,circle,fill=white] at (0,-1) (B3) {};
          \node[draw,circle,fill=white] at (0.5,-1) (B4) {};
          \node[draw,circle,fill=white] at (1,-1) (B5) {};
     	  \foreach \y in {B1,B2,B3,B4,B5} \draw[thick] (\y)--(A);
        \end{scope}};
      \foreach  \x in {(2,7),(3.5,7),(7,10),(10,13)}{
        \begin{scope}[shift=(\x)]
     	  \node[draw,circle,fill=white] at (0,0) (A) {};
          \node[draw,circle,fill=white] at (-0.5,-1) (B1) {};
          \node[draw,circle,fill=white] at (0,-1) (B2) {};
          \node[draw,circle,fill=white] at (0.5,-1) (B3) {};
     	  \foreach \y in {B1,B2,B3} \draw[thick] (\y)--(A);
        \end{scope}};
      \node[draw,circle,fill=orange] at (1,6) {};
      \node[draw,circle,fill=orange] at (1,2)  {};
      \foreach  \x in {1.5,2,2.5,3,3.5,4}{
        \begin{scope}[shift={(\x,6)}]
          \node[draw,circle,fill=blue] at (0,0) (A) {};
          \node[draw,circle,fill=blue] at (0,-4) (B) {};
          \draw[thick,dashed,blue] (B)--(A);
        \end{scope}};
      \draw[color=blue] (4.5,4) node[anchor=west]{pendent $(h+1)$-paths};
      \draw[color=blue] (0.5,6) node[anchor=east]{height $h$};
      \draw[color=blue] (0.5,2) node[anchor=east]{height $0$};
      \foreach  \x in {1,3,4,6,7}{
        \draw (0.5,6+\x) node[anchor=east]{height $h+\x$};
      };
      \foreach  \y in {2,5,8}{
        \draw (0.5,6+\y) node[anchor=east, color=orange]{height $h+\y$};
      };
    \end{tikzpicture}
    \caption{Inductive step. The tree shown is part of $T^{(h)}$, which has pendent paths with $h+1$ vertices as shown. Orange vertices at height $3j+h+2$ and one path $Q$ are independent invertible trees. For this illustration, just a few vertices are drawn on the left hand side of each level.}
    \label{fig:induction}
  \end{center}
\end{figure}
}

\title{Spectral arbitrariness for trees fails spectacularly}

\author[Fallat, Hall, Levene, Meyer, Nasserasr, Oblak, \v Smigoc]{Shaun M.~Fallat, H.~Tracy Hall, Rupert H.~Levene, Seth A.~Meyer, Shahla Nasserasr, Polona Oblak, Helena \v Smigoc}

\address[S.~M.~Fallat]{Department of Mathematics and Statistics, University of Regina, Regina, SK, S4S 0A2, Canada. (Corresponding Author)}
\email{shaun.fallat@uregina.ca}
\address[H.~T.~Hall]{Hall Labs, LLC}
\email{h.tracy@gmail.com}
\address[R.~H.~Levene and H.~\v Smigoc]{School of Mathematics and Statistics, University College Dublin, Belfield, Dublin 4, Ireland}
\email{rupert.levene@ucd.ie}
\email{helena.smigoc@ucd.ie}
\address[S.~A.~Meyer]{Mathematics Discipline, St. Norbert College, De Pere, WI 54115, USA}
\email{seth.meyer@snc.edu}
\address[S.~Nasserasr]{School of Mathematical Sciences, Rochester Institute of Technology, Rochester, NY, USA}
\email{shahla@mail.rit.edu}
\address[P.~Oblak]{Faculty of Computer and Information Science, University of Ljubljana, Ve\v cna pot 113, SI-1000 Ljubljana, Slovenia; Faculty of Mathematics and Physics, University of Ljubljana and Institute of Mathematics, Physics, and Mechanics, Jadranska ulica 19, 1000 Ljubljana, Slovenia}
\email{polona.oblak@fri.uni-lj.si}
\date{January 26, 2023} %

\begin{document}



\begin{abstract}
   If $G$ is a graph and $\ordm$ is an ordered multiplicity list which is realizable by at least one symmetric matrix with graph $G$, what can we say about the eigenvalues of all such realizing matrices for $\ordm$? It has sometimes been tempting to expect, especially in the case that $G$ is a tree, that any spacing of the multiple eigenvalues should be realizable. %
   In~\cite{MR2033476}, however, F. Barioli and S. Fallat produced the first counterexample: a tree on 16 vertices and an ordered multiplicity list for which every realizing set of eigenvalues obeys a nontrivial linear constraint.  
   
   We extend this by giving an infinite family of trees and ordered multiplicity lists whose sets of realizing eigenvalues are very highly constrained, with at most 5 degrees of freedom, regardless of the size of the tree in this family. In particular, we give the first examples of multiplicity lists for a tree which impose nontrivial nonlinear eigenvalue constraints and produce an ordered multiplicity list which is achieved by a \emph{unique} set of eigenvalues, up to shifting and scaling.
\end{abstract}
\maketitle

{\footnotesize \tableofcontents}

\section{Introduction}

Given a simple graph $G$ on $n$ vertices, the inverse eigenvalue problem for $G$ (\IEPG{}) asks
for all possible spectra of the associated symmetric matrices---that is, the possible multisets of eigenvalues of  $n\times n$ symmetric matrices with real entries whose off-diagonal entries have the same zero-nonzero pattern as the adjacency matrix of $G$.
Important subproblems include asking only for all possible combinations of eigenvalue multiplicities (possibly including the order in which they occur) or asking only for the maximum possible multiplicity.
For an overview of the literature on this problem and these subproblems, including discussion of how such inverse problems can arise naturally in studying the dynamics of a physical system for which $G$ encodes which pieces of the system interact with each other, see
the recent monograph \cite{HLS2022inverse}.
Inverse problems are, notoriously, as difficult as they are important, and the \IEPG{} and related inverse problems are no exception.
The considerable literature on this problem over the past few decades traces its roots to the less difficult (but still important) special case where $G$ is a path, %
for which the full \IEPG{} was resolved in 1974 by Hochstad \cite{MR382314} with the answer that any numerically ordered list of $n$ distinct real eigenvalues can be the ordered spectrum of a matrix whose pattern corresponds to a path. 
We say in this case, for $G$ a path on $n$ vertices, that the ordered multiplicity list of $n$ singletons is \emph{spectrally arbitrary}.
For a general graph $G$, spectral arbitrariness means that the particular numerical values of the various singleton or multiple eigenvalues %
can be chosen arbitrarily, with only the constraint that they occur in the proper order.
This property %
is very convenient whenever it happens to hold for all achievable ordered multiplicity lists, because it means that once you have solved the subproblem of ordered multiplicity lists, you have solved the entire \IEPG{}. 

For general $G$, not only the full \IEPG{} but even the subproblems of multiplicity lists or maximum multiplicity remain open and are considered to be very difficult.
Some progress has been made for small graphs or special families of graphs.
In particular, within the large family of graphs consisting of all trees, the maximum multiplicity has been shown \cite{MR1712856} to equal the path cover number of the tree.
Finding all possible ordered multiplicity lists for trees remains somewhat more difficult, but it was conjectured by Johnson and Leal Duarte \cite{MR1902112} that this would suffice for the full \IEPG{} for any given tree, or in other words that all achievable multiplicity lists for a tree would be spectrally arbitrary.
This was disproved by Barioli and Fallat \cite{BFH04} who gave an example of a relatively small tree and a particular ordered multiplicity list that can be realized for that tree, but only when the numerical placement of the eigenvalues satisfies a specific linear constraint.  This was extended by Ferrero et~al.~\cite{FFHHLMNS}, who generalized the tree in \cite{BFH04} to an infinite family of trees and showed that each has a multiplicity list that requires a linear constraint in eigenvalue placement.
However, prior to the present results, it was not known whether there were any graphs with multiplicity lists that required non-linear constraints, or to what degree spectral arbitrariness could fail for a tree. 

The present work first offers a construction (in Section~\ref{sec:path-construction}) that builds on the solution to the \IEPG{} for paths, but that produces matrices corresponding to a fairly general family of trees that is introduced and named \emph{hedges}. For a given hedge $T$, the construction produces a matrix that not only achieves the maximum possible multiplicity, but also has several other high eigenvalue multiplicities.
Special attention is paid to make the multiplicities of \emph{five} particular eigenvalues as high as possible, which exhausts all of the allowable choices within this technique and results in a matrix with several spectral and structural constraints.
In particular, while the construction allows for the eigenvalues corresponding to the five distinguished multiplicities to be chosen numerically with a full five degrees of freedom, once those five choices are made, the values for all other eigenvalues are completely determined, including many relatively high multiplicities in addition to the five.
This high degree of constraint is not surprising, since the construction has many fewer degrees of freedom than there are eigenvalues in the spectrum. 

Much more remarkable than the construction of a particular matrix with high multiplicities is its converse: We show (as Theorem~\ref{thm:main}) that
given some additional constraints on the hedge $T$ (it must be a \emph{lush hedge}),
any symmetric matrix with pattern $T$ that achieves at least these five distinguished multiplicities must arise from precisely this construction. The proof depends on a combinatorial argument that extends the idea of the path cover number and the idea of zero forcing, and that uses the structure of the introduced family of hedges in an essential way.

The converse result demonstrates that spectral arbitrariness for trees can fail to a much greater degree than was imagined in the work that first demonstrated spectral arbitrariness for trees failing by a single linear constraint.
At least two degrees of freedom, called \emph{shifting and scaling,} are always available in the placement of eigenvalues with particular multiplicities.
If $A$ is a matrix exhibiting the pattern of a graph $G$, then for any chosen strictly positive real number $m$ and any real number $b$, the matrix $mA + bI$, whose spectrum has been scaled by $m$ and shifted by $b$, will have the same pattern $G$ and the same ordered list of eigenvalue multiplicities. 
For a list of $k$ multiplicities, then, spectral arbitrariness implies two trivial and $k - 2$ non-trivial degrees of freedom.
A single linear constraint reduces this to $k - 3$ non-trivial degrees of freedom.
Under the hypotheses of Theorem~\ref{thm:main}, the $k - 3$ non-trivial degrees of freedom reduce to only three non-trivial degrees of freedom even though $k$ is quadratic in the diameter of the lush hedge $T$.
The many additional constraints include non-linear constraints, including a cubic constraint described explicitly in Example~\ref{ex:T31-M-construction}.

In Section~\ref{sec:rigid} the failure of spectral arbitrariness is pushed to its final collapse by exhibiting a certain ordered multiplicity list, for any lush hedge of sufficient height, whose eigenvalue placement is entirely immune to non-trivial perturbation.
To achieve this, particular choices are made for the five distinguished eigenvalues which use up the three remaining non-trivial degrees of freedom to engineer a further three eigenvalue coincidences.
The resulting list of eigenvalue multiplicities
has only the two trivial degrees of freedom from shifting and scaling in its realization space, which means that the relative spacing of the eigenvalues is completely rigid.

The work outlined above has important consequences in our understanding of the inverse eigenvalue problem for graphs.
It allows us in particular (in Section~\ref{sec:S(T)}) to resolve certain open questions and established conjectures,
and we believe that the combinatorial techniques developed will find significant use in the field.

\section{Setup and notation}

In this section we introduce some notation and terminology and outline how this work fits into the existing body of results on the \IEPG{}. 

\subsection{Standard notation}
For any positive integer $n$, let $[n]:=\{1,\ldots,n\}$, and let $[0]:=\emptyset$. We write ${\bf 0}$ for a zero vector, ${\bf 1}$ for an all-ones vector, and ${\bf e}_i$ for a vector with all entries equal to $0$ except the $i$th entry, which is $1$. The order of vectors will be clear from the context.  A vector ${\bf t}=(t_i)\in \bR^n$ is \emph{positive} if $t_i>0$ for all $i\in [n]$. We denote the set of all $m\times n$ matrices with real entries by $\bR^{m\times n}$. By $I_n$ we denote the identity matrix in $\bR^{n\times n}$, and $E_{i,j}$ denotes the matrix unit of the appropriate order, with a $1$ in the $(i,j)$ entry, and $0$s elsewhere. A vector or a matrix is \emph{nowhere zero} if each of its entries is nonzero. The characteristic polynomial of a matrix $A$ is denoted by $p_A(x)$. 

 We will consider only simple graphs $G=(V(G),E(G))$, writing $|G|:=|V(G)|$.
The \emph{degree} of a vertex $v$ in a graph $G$ is denoted by $\deg_G(v)$; a \emph{leaf} is a vertex of degree one. If $u$ and $v$ are vertices of $G$, then a \emph{path} from $u$ to $v$ is a sequence of distinct vertices $u=v_0,v_1,\dots,v_k=v$ where $k$ is a non-negative integer, and $\{v_i,v_{i+1}\}$ is an edge of $G$ for $0\le i<k$; the \emph{length} of this path is $k$~edges. The \emph{distance} between $u$ and $v$, denoted by %
$d(u,v)$, is the minimum length of a path between $u$ and $v$, and the \emph{diameter} of~$G$ is the maximum distance between two vertices in~$G$. 
A set of vertex-disjoint induced paths in~$G$ whose vertex union is $V(G)$ is called a \emph{path cover} of $G$. The minimum number of paths in a path cover is called the \emph{path cover number} of $G$ and is denoted by $P(G)$.  

In this work, the letter $T$ will denote a \emph{tree}, which is an acyclic, connected, %
finite graph.  A \emph{forest} is a disjoint union of trees.
A \emph{rooted tree} $T$ is a tree with a distinguished vertex $r$, called the \emph{root} of the tree. 
In a rooted tree with at least two vertices, a \emph{leaf} is a non-root vertex of degree 1. As we will often refer to paths, for $n\ge1$ we write $P_n$ for the path $1\text{---}2\text{---}\cdots\text{---}n$ with vertex set $[n]$. We will consistently regard this as a rooted tree, with root vertex $r=1$, and when $T=P_1$, the only vertex is considered both the root and a leaf.
 For $n\ge1$, the star graph $K_{1,n}$ is the rooted tree on $n+1$ vertices consisting of a root vertex connected to $n$ leaves.
 
 If $y$ and $z$ are neighbouring vertices of~$T$ so that the path from the root to $y$ goes through $z$,  then we say that $y$ is a \emph{child} of $z$, or that $z$ is the (necessarily unique) \emph{parent} of $y$.  Two vertices with the same parent are called \emph{siblings}. An induced path $Q$ on $k$ vertices is \emph{a \pendentpath{k} in~$T$} if the root of~$T$ is not in $Q$ and the only edge between $Q$ and $T\setminus Q$ connects to one of the ends of $Q$. A \emph{pendent path in~$T$} is a \pendentpath{k} in~$T$, for some $k\ge1$.

\subsection{The \texorpdfstring{\IEPG{}}{IEP-G} and spectral arbitrariness}
Given a graph $G$ on $n$ vertices, let $\S(G)$ be the set of all $n \times n$ real symmetric  matrices $A=(a_{i,j})$ such that for $i \ne j$, $a_{i,j}\ne 0$ if and only if $\{i,j\} \in E(G)$, with no restriction on the diagonal entries of~$A$. (Here we implicitly identify $V(G)$ with $[n]$.) The \emph{inverse eigenvalue problem} of $G$ (\IEPG{}) asks which multisets of $n$ real numbers are the spectra of some matrix in $\S(G)$. The \IEPG{} is a difficult open problem, and to date it has been resolved for some families of graphs only: for graphs on at most five vertices,~\cite{IEPG2, MR3291662}, paths~\cite{MR447294, MR447279, MR382314}, cycles~\cite{MR583498}, generalized stars~\cite{MR2022294}, linear trees~\cite{MR4357320}, complete graphs~\cite{MR3118942} and lollipop and barbell graphs~\cite{MR4316738}. It is the topic of a large body of work and has motivated the investigation of several subproblems, see e.g.~\cite{HLS2022inverse}. %

For a matrix $A\in \bR^{n\times n}$, we use  $\mult(\lambda,A)$ to denote the multiplicity of a scalar $\lambda$ as an eigenvalue. If $A$ has $k$ distinct eigenvalues $\lambda_1, \ldots,\lambda_{k}$ we denote its spectrum as the multiset $\spec(A)=\{\lambda_1^{(m_1)},\ldots,\lambda_k^{(m_k)}\}$, where $m_i=\mult(\lambda_i,A)\ge1$ for $i\in[k]$. We say that a multiset $\sigma$ is \emph{realizable} for $G$ if $\sigma=\spec(A)$ for some $A\in \S(G)$.  The \emph{maximum multiplicity} of a graph $G$ is $M(G):=\max\{\mult(\lambda,A):\lambda\in \sigma(A),A\in\S(G)\}$.

The list $\{m_1,\ldots,m_k\}$ is called the \emph{unordered multiplicity list} of~$A$, and if $\lambda_1< \ldots<\lambda_{k}$ in the notation above, then $\ordm(A)=(m_1,\ldots,m_k)$ is called the \emph{ordered multiplicity list} of~$A$.\label{def:ord-mult-list}  We say that the unordered multiplicity list $\{m_1,\ldots,m_k\}$ is \emph{realizable} in $\S(G)$ if there exists a set of distinct values $\lambda_1,\ldots,\lambda_k$ such that $\spec(A)=\{\lambda_1^{(m_1)},\ldots,\lambda_k^{(m_k)}\}$ for some $A\in \S(G)$. We say that an ordered list $\ordm=(m_1,\ldots,m_k)$ is \emph{realizable} for $G$ if $\ordm=\ordm(A)$ for some $A\in \S(G)$.

An ordered multiplicity list $\ordm=(m_1,\dots,m_k)$ is said to be \emph{spectrally arbitrary for $G$} if, for all real numbers $\lambda_1,\dots,\lambda_k$ satisfying $\lambda_1<\dots<\lambda_k$, the multiset $\{\lambda_1^{(m_1)},\ldots,\lambda_k^{(m_k)}\}$ is realizable for~$G$.
Moreover, $G$ is called \emph{spectrally arbitrary} if every ordered multiplicity list $\ordm$ that is realizable for~$G$ is spectrally arbitrary for~$G$.

We adopt natural operations on multisets; for example, a union of multisets is formed by adding multiplicities, an intersection corresponds to taking the minimum of multiplicities, (multi)set differences are formed by subtracting multiplicities and taking the non-negative part, and if $s\geq 0$ is an integer and $\Lambda=\{\lambda_1^{(m_1)},\ldots,\lambda_k^{(m_k)}\}$ is a multiset, we use the notation $s\Lambda:=\{\lambda_1^{(sm_1)},\ldots,\lambda_k^{(sm_k)}\}$. %

The questions of characterizing all possible ordered or unordered realizable multiplicity lists for a given graph have been studied extensively, with many advances noted primarily for trees,~\cite{MR4284782, BFH04, MR2547901, MR3557827, FFHHLMNS,MR1712856, MR1902112,KimShader, MR1899084}.  In particular,  upper bounds for the sums of largest multiplicities have been computed~\cite{FFHHLMNS,MR4295979}. Additionally, in Section \ref{sec:comb}, we will use that for any tree $T$, we have the equality $M(T)=P(T)$  which was first shown in \cite{MR1712856}.  

It has been recently shown that linear trees are spectrally arbitrary~\cite{MR4357320}.
On the other hand, it was first proved by Barioli and Fallat~\cite{BFH04} that in general, a tree need not be spectrally arbitrary. Following their idea, it is possible to prove that the tree $\bftree$ shown in Figure~\ref{fig:BFtree} is the smallest such example and the ordered multiplicity list $\ordm_{BF}=(1,2,4,2,1)$ is realizable but not spectrally arbitrary for $\bftree$. Failure of spectral arbitrariness is due to the fact that the corresponding eigenvalues $\lambda_1<\ldots<\lambda_5$ realizing $\ordm_{BF}$ in $\S(\bftree)$ must fulfill the linear constraint $\lambda_1+\lambda_5=\lambda_2+\lambda_4$. In Corollary~\ref{cor:four-multiplicities} we prove that the same linear constraint must be fulfilled for a large family of trees with four eigenvalues of high multiplicities.
Working with a family of larger trees $T$, Ferrero et~al.~\cite{FFHHLMNS} found an unordered multiplicity list (depending on $T$) so that all realizing matrices for $T$ satisfy another linear constraint. As we will see in Example~\ref{ex:T31-M-construction} and Section~\ref{sec:rigid}, our results allow us to recover, refine and greatly extend these eigenvalue constraints.

 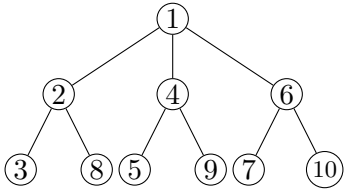
\begin{figure}[htb]
    \centering{%
    \begin{tikzpicture}[scale=0.5]
   \node[draw,circle,fill=white] at (0,0) (D) {$1$};
    \node[draw,circle,fill=white] at (-3,-2) (A) {$2$};
    \node[draw,circle,fill=white] at (-4,-4) (E) {$3$};
    \node[draw,circle,fill=white] at (0,-2) (B) {$4$};
    \node[draw,circle,fill=white] at (-1,-4) (G) {$5$};
    \node[draw,circle,fill=white] at (3,-2) (C) {$6$};
    \node[draw,circle,fill=white] at (2,-4) (I) {$7$};
    \node[draw,circle,fill=white] at (-2,-4) (F) {$8$};
    \node[draw,circle,fill=white] at (1,-4) (H) {$9$};
    \node[draw,circle,fill=white] at (4,-4) (J) {\footnotesize$10$};
        \draw(D)--(A)--(E);
        \draw(B)--(G); 
        \draw(C)--(I);
        \draw (H)--(B)--(D)--(C)--(J);
        \draw (A)--(F);
    \end{tikzpicture}}
    \caption{The Barioli-Fallat tree $\bftree$.}\label{fig:BFtree}
\end{figure}

\subsection{Geometric point of view}
In order to quantify the degree to which spectral arbitrariness fails for a particular graph $G$ and an ordered multiplicity list $\ordm$, it will be useful to define a moduli space $\RS(G, \ordm)$, see Definition~\ref{def:RS}, that captures essential variations in the possible placement of eigenvalues, where ``essential'' means that we wish to ignore trivial modifications of the spectrum by orientation-preserving affine transformations (also called ``shifting and scaling''). 

Given a symmetric matrix $A \in \S(G)$ realizing some ordered multiplicity list $\ordm$, and $m,b\in \bR$ with $m>0$,
the matrix $mA + bI \in \S(G)$ gives a different placement of eigenvalues for the same ordered multiplicity list~$\ordm$.
What this affine transformation does not change is the \emph{relative spacing} of eigenvalues, which is to say the ratios between eigenvalue gaps. The notation $\RS(G, \ordm)$ is meant to suggest 
every possible ``relative spacing'' of eigenvalues for $G$ and $\ordm$. 

The standard simplex of dimension $d$ in $\mathbb{R}^{d+1}$ will be denoted $\Delta_d$. Its
interior, consisting of every %
$(d+1)$-tuple of strictly positive real numbers whose sum is $1$, will be denoted $\simplex_d$.
Given a specific spectrum $\sigma$ realizing all $k$ multiplicities of $\ordm$,
we obtain a point in $\simplex_{k-2}$ by rescaling $\sigma$ to make the total width $1$ and then by listing, from left to right, the $k - 1$ rescaled gaps between successive distinct eigenvalues,
which are positive and sum to~$1$.

\begin{definition}
\label{def:RS}
  Let $G$ be a graph on $n$ vertices and $\ordm = (m_1, \dots, m_k)$ an ordered multiplicity list of positive integers with $\sum_{i\in[k]} m_i = n$. 
  In the case $k \ge 2$, the moduli space $\RS(G, \ordm)$ is defined as the subset of $\simplexk$ such that
  ${\bf p} \in \RS(G, \ordm)$ if and only if there exists a matrix $A \in \S(G)$
  with spectrum $\{\lambda_1^{(m_1)}, \dots, \lambda_k^{(m_k)}\}$, $\lambda_1<\ldots<\lambda_k$,   and
  \[
  {\bf p} = \left( \frac{\lambda_2 - \lambda_1}{\lambda_k - \lambda_1},
  \dots,
  \frac{\lambda_k - \lambda_{k - 1}}{\lambda_k - \lambda_1} \right).
  \]
  In the case $k = 1$, $\RS(G, \ordm)=\RS(G,(n))$ is defined as a single point when $G$ has no edges; otherwise $\RS(G, (n))$ is defined to be the empty set.
\end{definition}

Observe that if ${\bf p}=(p_1,\dots,p_{k-1})\in \RS(G,\ordm)$, then for $\lambda_1:=0$ and $\lambda_j:=\sum_{i=1}^{j-1}p_{i}$ for $2\le j\le k$, we have $0=\lambda_1<\lambda_2<\dots<\lambda_k=1$, and the spectrum $\{\lambda_1^{(m_1)}, \dots, \lambda_k^{(m_k)}\}$ is realized by some matrix in $\S(G)$. Moreover, up to translation and scaling, all such $\lambda_1<\dots<\lambda_k$ arise in this way. Hence,
$\ordm$ is spectrally arbitrary for $G$ if and only if
$\RS(G, \ordm) = \simplexk$, and $G$ is spectrally arbitrary if and only if, for every ordered partition $\ordm$ of $n$ with $k \ge 2$ parts, $\RS(G, \ordm)$ is either empty or equal to $\simplexk$. 

\begin{example}\label{ex:BF-12421}
    Recall that for $\bftree$ and the ordered multiplicity list $\ordm_{BF}=(1,2,4,2,1)$ the corresponding eigenvalues $\lambda_1<\ldots<\lambda_5$ realizing $\ordm_{BF}$ in $\S(\bftree)$ must fulfill the linear constraint $\lambda_1+\lambda_5=\lambda_2+\lambda_4$, or equivalently, $\lambda_2-\lambda_1=\lambda_5-\lambda_4$. In fact, there are no other constraints, so it follows that 
$$ \RS(\bftree, \ordm_{BF})=\left\{(p_1,p_2,p_3,p_4)\in \simplex_3:p_1=p_4\right\},$$
 which implies that $\RS(\bftree, \ordm_{BF})$ is 2-dimensional, being the intersection of the interior of the $3$-dimensional standard simplex $\simplex_3$ with a linear subspace, as shown in Figure~\ref{fig:RS-BFtree}.
 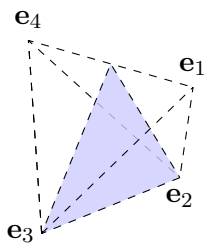
\begin{figure}[htb]
 \begin{tikzpicture}[line join=bevel,3d view={5}{-20}]
    \coordinate[label=$\mathbf{e}_1$](A) at (1,1,1);
  \coordinate[label=below:$\mathbf{e}_2$](B) at (1,-1,-1);
  \coordinate[label=left:$\mathbf{e}_3$](C) at (-1,1,-1);
  \coordinate[label=$\mathbf{e}_4$](D) at (-1,-1,1);
  \coordinate(E) at (0,0,1);
\draw[dashed] (C)--(D)--(A);
\draw[dashed] (C)--(D)--(B);
  \fill[blue!20,opacity=.8](C)--(E)--(B);
  \draw[dashed](B)--(C)--(E)--(B);
    \draw[dashed] (C)--(A)--(B);
\end{tikzpicture}
     \caption{The two dimensional manifold $\RS(\bftree, \ordm_{BF})$, in blue, shown inside the $3$-dimensional standard simplex $\Delta_3$.}
     \label{fig:RS-BFtree}
 \end{figure}
\end{example}

In this work we aim to better understand the geometry of $\RS(T, \ordm)$ for a specific family of trees, which we introduce in the next subsection. In contrast to previously known examples, we will show that $\RS(T, \ordm)$ does not have to be convex (see Example~\ref{ex:non-convexity}). Moreover, in Section~\ref{sec:rigid} we will see that it is possible to arrange for $\RS(T,\ordm)$ to consist of a single point for certain trees $T$ of arbitrarily large size, representing an extreme failure of spectral arbitrariness.
    
\subsection{Hedges}\label{subsec:hedges}
We will work with the following special family of rooted trees.

\begin{definition}
A rooted tree $T$ is called a \emph{hedge} if either $T=P_1$ or every leaf of~$T$ has the same distance to the root of~$T$.  
\end{definition}

The \emph{height} of a vertex $u$ in a hedge $T$, written as $\h(u)$, is the shortest distance from $u$ to a leaf, and the \emph{height} of a hedge, denoted by $\h(T)$, is the height of the root, $\h(r)$. For example, the height of $P_{n+1}$ is $n$, and $K_{1,n}$ has height~$1$. 

When $T$ is a hedge of height $H$ and $i\ge 0$, we define $V_i(T)$ to be the set of vertices of height $i$ in~$T$. In particular, $V_{H}(T)$ contains only the root vertex of~$T$, and 
$V_0(T)$ is the set of leaves of~$T$ . Clearly, $V_i(T)$ is non-empty for $0\le i\le H$, and $V_i(T)=\emptyset$ otherwise. For $i \geq 1$ let
  $$\ell_i(T):=|V_{i-1}(T)|-|V_{i}(T)|.$$ 
Notice that $\ell_{H+1}(T)=1$ and $\ell_i(T)=0$ if $i>H+1$. Since any non-leaf vertex of~$T$ has at least one child (unless we are in the trivial case $T=P_1$), we have $\ell_i(T)\ge0$ for all $i\in [H]$. When context allows we write $V_i:=V_i(T)$ and $\ell_i:=\ell_i(T)$.

Note that the $\ell_i$ sequence can be thought of as encoding the amount of branching at each height; see Figure~\ref{fig:H=2} for some illustrative examples.  
Equivalently, $T$ can be constructed by starting with $P_{H+1}$ and recursively forming the disjoint union with $\ell_i$ \pendentpath{i}s, and connecting each of these to the a vertex at height $i$ with a new connecting edge, %
in the order $i=H, H-1, H-2,\ldots, 1$.

\begin{figure}[htb]\label{fig:lush and non-lush}
  \definecolor{green_}{rgb}{0,.9,.4} %
  \definecolor{c-12}{rgb}{0,0.1,.9} %
  \definecolor{cr}{rgb}{.9,.2,0} %
\begin{subfigure}[b]{0.44\textwidth}    \begin{center}
     \scalebox{.4}{
     \begin{tikzpicture}
 	 \node[draw,circle,fill=white] at (-4,-4) (A) {\phantom{$1$}};
	 \node[draw,circle,fill=white] at (0,-4) (B) {\phantom{$1$}};
 	 \node[draw,circle,fill=white] at (4,-4) (C) {\phantom{$1$}};
 	 \node[draw,circle,fill=white] at (-16/3,-8) (E) {\phantom{$1$}};
 	 \node[draw,circle,fill=green_] at (-8/3,-8) (F) {\phantom{$1$}};
 	 \node[draw,circle,fill=white] at (-4/3,-8) (G) {\phantom{$1$}};
 	 \node[draw,circle,fill=green_] at (4/3,-8) (H) {\phantom{$1$}};
 	 \node[draw,circle,fill=white] at (8/3,-8) (I) {\phantom{$1$}};
 	 \node[draw,circle,fill=green_] at (16/3,-8) (J) {\phantom{$1$}};
       \draw[color=c-12,very thick] (0,0)--(A)--(E);
     \draw[very thick,color=cr](B)--(G);
     \draw[very thick,color=cr](C)--(I);
     \draw[dashed] (H)--(B);
     \draw[dashed] (B)--(0,0)-- (C);
     \draw[dashed] (C)--(J);
     \draw[dashed] (A)--(F);
   \node[draw,circle,fill=white] at (0,0) (0) {\phantom{$1$}};
    \end{tikzpicture}}
     \caption{The lush hedge $\bftree$.}
     \label{fig:BFtree2}
    \end{center}
    \end{subfigure}
\begin{subfigure}[b]{0.44\textwidth}
\begin{center}
     \scalebox{.4}{
     \begin{tikzpicture}
 	 \node[draw,circle,fill=white] at (-4,-4) (A) {\phantom{$1$}};
	 \node[draw,circle,fill=white] at (0,-4) (B) {\phantom{$1$}};
 	 \node[draw,circle,fill=white] at (4,-4) (C) {\phantom{$1$}};
 	 \node[draw,circle,fill=white] at (-16/3,-8) (E) {\phantom{$1$}};
 	 \node[draw,circle,fill=green_] at (-8/3,-8) (F) {\phantom{$1$}};
 	 \node[draw,circle,fill=white] at (-4/3,-8) (G) {\phantom{$1$}};
 	 \node[draw,circle,fill=green_] at (4/3,-8) (H) {\phantom{$1$}};
 	 \node[draw,circle,fill=white] at (12/3,-8) (I) {\phantom{$1$}};
 	 \node[draw,circle,fill=green_] at (-12/3,-8) (J) {\phantom{$1$}};
       \draw[color=c-12,very thick] (0,0)--(A)--(E);
     \draw[very thick,color=cr](B)--(G);
     \draw[very thick,color=cr](C)--(I);
     \draw[dashed] (H)--(B);
     \draw[dashed] (B)--(0,0)-- (C);
     \draw[dashed] (A)--(J);
     \draw[dashed] (A)--(F);
   \node[draw,circle,fill=white] at (0,0) (0) {\phantom{$1$}};
    \end{tikzpicture}}
     \caption{A hedge which is not lush.
     }
     \label{fig:non-lush-hedge}
    \end{center}
\end{subfigure}
   \subfloat{\scalebox{.6}{
        \begin{tikzpicture}
 	       \node[draw=none] at (0,-6) {\phantom{$\ell$}};
 	       \node[draw=none] at (0,-5.4) {$|V_0|=6$};
 	        \node[draw=none] at (0,-4) {$\ell_1=3$};
 	        \node[draw=none] at (0,-2.7) {$|V_1|=3$};
 	        \node[draw=none] at (0,-1.35) {$\ell_2=2$};
 	        \node[draw=none] at (0,-0) {$|V_2|=1$};
 	        \node[draw=none] at (0,1.35) {$\ell_3=1$};
        \end{tikzpicture}}
     }
\caption{Two hedges with $H=2$ and $\ell_1=3,\ell_2=2$ and $\ell_3=1$. The solid edges show spanning subgraphs both isomorphic to $3{\color{green_!70!black}P_1}\cup 2{\color{cr!70!black}P_2}\cup 1{\color{c-12!70!black}P_3}$. The dashed edges encode the amount of branching and the number of dashed edges from height $h$ to height $h-1$ is equal to $\ell_h$ for $h\in [H]$.}\label{fig:H=2}
\end{figure}
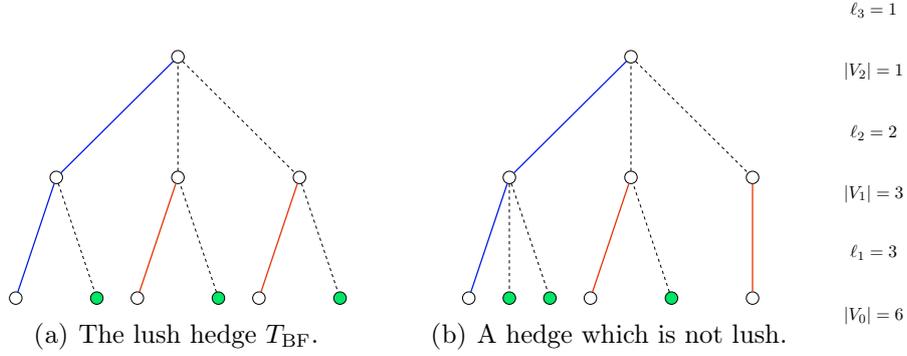

\begin{definition}
    A hedge $T$ is called \emph{lush} %
    if any vertex
    of height at least~$2$ has at least three children, and any vertex of height~$1$  has at least two children. In other words, for every non-leaf vertex $v$ of~$T$, we have
    $$\deg_T(v) \geq \begin{cases}
     3, & \text{if } \h(v)=1 \text{ or } \h(v)=H,\\
     4, & \text{if $H>2$ and } 2\leq \h(v) < H.
    \end{cases}$$
\end{definition} 
We illustrate these definitions with some examples.
\begin{example}
The path $P_n$ with $n\geq 2$ is a hedge of height $n-1$, with $\ell_n(P_n)=1$ and $\ell_i(P_n)=0$ for all $i\in [n-1]$.  It is not lush as $\deg_{P_n}(u)\leq 2$ for all $u\in V(P_n)$. The graph $P_1$ is a lush hedge.

A perfect binary tree %
with height $H\geq 2$ is a hedge having $\ell_i=|V_i|=2^{H-i}$ for $i\in [H]$. It is not lush as the degree of the root is only 2.

A hedge of height $1$ on $n$ vertices is the star graph $K_{1,n-1}$, and it is lush as long as $n\geq 3$.

The Barioli-Fallat tree $\bftree$, shown in Figures~\ref{fig:BFtree} and~\ref{fig:BFtree2}, is a lush hedge of height 2.  Moreover, $\bftree$ is the smallest lush hedge of height $2$, and has $\ell_1=3$, $\ell_2=2$ and $\ell_3=1$. %
\end{example}

A rooted tree $L$ of height $H\ge2$ is a lush hedge if and only if there exist $t \geq 3$ and lush hedges $L_1,\ldots,L_t$, all of height $H-1$, such that the root vertex of $L$ is adjacent to the root vertex of $L_i$ for $i=1,\dots,t$. Thus, each such lush hedge $L$ can be constructed recursively; see Figure~\ref{fig:LushHedge}. 

\begin{figure}[htb]
    \begin{tikzpicture}
   \node[draw,circle,fill=white] at (0,0) (D) {$r$};
     \draw (-1,-2)--(D)--(1,-2);
     \draw (-4,-2)--(D)--(4,-2);
     \draw (-2,-2)--(D)--(2,-2);
     \draw[fill=white] (-4,-3) ellipse (0.9cm and 1cm);
     \draw[fill=white] (-2,-3) ellipse (0.8cm and 1cm);
     \draw[fill=white] (2,-3) ellipse (0.6cm and 1cm);
     \draw[fill=white] (4,-3) ellipse (0.5cm and 1cm);
     \node[anchor=south,draw=none] at (-4,-3.5) {$L_1$};
     \node[anchor=south,draw=none] at (-2,-3.5) {$L_2$};
     \node[anchor=south,draw=none] at (2,-3.65) {$L_{t-1}$};
     \node[anchor=south,draw=none] at (4,-3.5) {$L_t$};
     \node[anchor=south,draw=none] at (0,-3.5) {$\cdots$};
    \end{tikzpicture}
    \caption{A lush hedge $L$ of height $H\geq 3$ with root vertex~$r$ and $\deg_L(r)=t$. Subtrees $L_i$, $i\in[t]$, are all lush hedges of height $H-1$.
    }\label{fig:LushHedge}
 \end{figure}
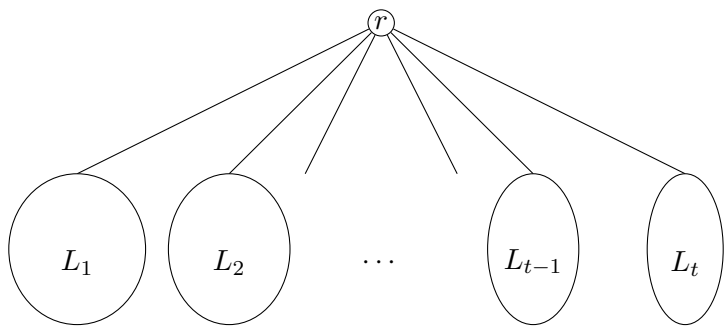

We now observe for future use that the high degree of branching required for a lush hedge implies that the values of $\ell_i$ are rapidly decreasing. 
\begin{lemma}\label{lem:decreasing-ells}
For every lush hedge $T$ with height $H$ we have
    \begin{equation*}
    \ell_i(T)\geq %
    2\sum\limits_{j = i + 1}^{H + 1} \ell_j(T)
\end{equation*}
for all $i\geq 2$.
\end{lemma}
\begin{proof}
  Since $T$ is lush, for $i\ge2$ we have $|V_{i-1}(T)|\ge 3|V_i(T)|$. Hence, $\ell_i(T)\ge 2|V_i(T)|=2\sum_{j=i+1}^{H+1} \ell_j(T)$.
\end{proof}

\subsection{Matrix sets and weights} When discussing the uniqueness of realizations in $\S(T)$ implied by partial (or complete) information about the  eigenvalues, we cannot avoid the inherent ambiguity arising from diagonal similarity. While only conjugation by diagonal matrices with $1$ or $-1$ on the diagonal preserves symmetry, conjugation by any invertible diagonal matrix preserves combinatorial symmetry in the following sense. 

A matrix $A=\begin{pmatrix} a_{i,j} \end{pmatrix}\in \bR^{n \times n}$ is called \emph{combinatorially symmetric} if $a_{i,j} \neq 0$ implies $a_{j,i} \neq 0$ for all $i \neq j$, see e.g.~\cite{MR3857534,Maybee}. Given an $n \times n$ combinatorially symmetric matrix $A$, we let $G(A)$ be the graph with vertex set $[n]$ and edges $\{i,j\}$ whenever $i \neq j$ and $a_{i,j} \neq 0$. The graph $G(A)$ is referred to as the \emph{graph of~$A$}. We will consider combinatorially symmetric matrices that satisfy $a_{i,j}a_{j,i}>0$, as those are similar to a symmetric matrix under diagonal similarity. This can be easily established by induction and can also be deduced indirectly from the work in \cite{MR3716248}.   We %
denote the set of all such combinatorially symmetric matrices corresponding to a tree $T$ by
\begin{align*}
  \cs(T) &:= \{A\in \bR^{|T|\times |T|} \colon G(A)=T,\; {{a_{i,j}a_{j,i}>0\;\, \textrm{if}\, \{i,j\}\in E(T)}}\}.
\end{align*}
 For most of this paper we will work with matrices in $\cs(T)$ rather than $\S(T)$. Of course, since $\S(T)\subseteq \cs(T)$, the results we establish for $\cs(T)$ also apply to $\S(T)$.

Given a tree $T$, two matrices $A=(a_{i,j}), B=(b_{i,j}) \in \cs(T)$ are similar under a diagonal similarity if and only if  $a_{i,i}=b_{i,i}$  and $a_{i,j}a_{j,i}=b_{i,j}b_{j,i}$ for all $\{i, j\} \in E(T)$. 
Diagonal similarity is an equivalence relation, and uniqueness of realizations considered in this work will be defined up to an equivalence class for diagonal similarity. This inherent non-uniqueness in $\cs(T)$ leads us to consider the following objects.
We call a function \[ w\colon V(T)\cup E(T)\to \bR,\quad {w(e)> 0}\text{ for }e\in E(T)\]
a \emph{weight function} on $T$, and write $W(T)$ for set of all such weight functions on $T$. (Note that the weight $w(i)$ of a vertex $i\in V(T)$ is unconstrained, so can be any real number.) We slightly abuse notation by writing $w(i,j)$ instead of $w(\{i,j\})$ when $\{i,j\}\in E(T)$. Given $w\in W(T)$, consider the set
\begin{align*}
  \ns(T):= \{A=\begin{pmatrix}
    a_{i,j}
  \end{pmatrix}\in \cs(T)\colon & a_{i,j}a_{j,i}=w(i,j),\,\{i,j\}\in E(T),\;\\
  &a_{i,i}=w(i),\; i\in V(T)\}.
\end{align*}
Note that any matrix $A\in \cs(T)$ lies in $\ns(T)$, for a unique weight $w$ on $T$. Indeed, $\ns(T)$ is precisely the equivalence class of $\R(T)$ (with respect to the equivalence relation of diagonal similarity) which contains $A$. In this case we say that $w_A:=w$ is \emph{the weight of~$A$}.
For simplicity of exposition, we will often define an equivalence class through a representative satisfying  $a_{i,j}=1$ or $a_{j,i}=1$ for all $\{i, j\} \in E(T)$, and when convenient we will assume $a_{i,j}=1$ if $\{i,j\} \in E(G)$ and $i >j$.  

For $A\in \cs(T)$  and a subgraph $T_0\subseteq T$ with vertex set $V_0$, the matrix $A[T_0]=A[V_0]$ is the principal submatrix of~$A$ whose rows and columns are indexed by $V_0$. Similarly, if ${\bf v}\in\mathbb{R}^{|T|}$ has entries labelled by $V(T)$, then ${\bf v}[T_0]={\bf v}[V_0]\in\mathbb{R}^{|T_0|}$ is the vector whose entries are indexed by $V_0$.

\section{Duplicating and collapsing branches, and the path-to-hedge construction}\label{Branching}

Duplicating and collapsing branches of graphs are opposite operations that  expand and contract graphs. 
If these operations can be performed so that the eigenvalues of matrices corresponding to the expanded graph %
and the original graph coincide, we can construct a matrix for the expanded graph that has eigenvalues with high multiplicities. 
Theorem~\ref{thm:arbitrary M} uses this technique to construct a large matrix~$A$, whose graph is a hedge, from a matrix~$C$ whose graph is a path, so that $\sigma(A)$ consists of copies of the spectra of various principal submatrices of $C$. %
We call this the path-to-hedge construction.
At the end of the section, we elaborate these operations by examples. 

\subsection{Duplicating and collapsing branches in trees}\label{sec:branching-collapsing}

We begin with branching operations on trees and on symmetric matrices corresponding to tree. These were introduced in~\cite{JS08}. 

Let $T$ be a tree and $v$ a vertex of~$T$. We call a connected component of $T\setminus v$ a \emph{branch of~$T$ at $v$}. Let $B_0$ be a branch of~$T$ at $v$ and let $\{v,b_0\}$ be the edge joining $v$ to $B_0$. Recall~\cite{JS08} that for $s\ge1$, \emph{the $s$-branch duplication of $B_0$ at $v$ in~$T$} is the tree $T'$ obtained from $T$ by adding $s$ disjoint copies $B_1,\dots,B_s$ of $B_0$ and the edges $\{v,b_i\}$ for $i\in [s]$, where $b_i\in V(B_i)$ is the copy of $b_0$.

Given $T,B_0,v,b_i,s, T'$ as above and $A=\begin{pmatrix}
  a_{i,j}
\end{pmatrix}\in \S(T)$, for any nowhere-zero vector ${\bf z}=(z_0,\dots,z_s)\in \bR^{s+1}$ with ${\bf z}^\top{\bf z}=1$, we can form a matrix $A'=\begin{pmatrix}  a'_{i,j}
\end{pmatrix}\in \S(T')$, called the \emph{$s$-summand duplication of $A[B_0]$ at $v$ relative to ${\bf z}$}, by defining \[\text{$A'[T\setminus B_0]=A[T\setminus B_0]$, $A'[B_i]=A[B_0]$, $a'_{v,b_i}=z_i a_{v,b_0}$ for $0\le i\le s$.}\] As shown in~\cite{JS08}, the characteristic polynomials of these matrices are related by $p_{A'} = p_A \cdot (p_{A[B_0]})^s$. 
Hence, as multisets, we have
\begin{equation}\label{eq:branching-spectrum}
    \spec(A')=\spec(A)\cup s \spec(A[B_0]).
\end{equation} 
As explained above, we will work with matrices in $\cs(T)$, which need not be symmetric, and their weights $w\in W(T)$. We define a more general branching operation in this context, as follows. 

\begin{definition}\label{def:weight-summand-duplication}
Let $T,B_i,v,b_i,s,T'$ be as above, $w\in W(T)$ a weight function on $T$, and  ${\bf t}=(t_0,\dots,t_s)\in \bR^{s+1}$ a positive  vector with $t_0+\dots+t_s=1$.
A weight function $w' \in W(T')$ is  \emph{the $s$-summand duplication of $w$ 
at $v$ for $B_0$ relative to ${\bf t}$} if: 
\begin{itemize}
    \item  $w'$ agrees with $w$ on $V(T)\cup E(T)\setminus \{\{v,b_0\}\}$;
    \item  on $B_i$, $w'$ agrees with $w$ on $B_0$, once the vertices and edges of $B_i$ are identified with those of $B_0$; and
    \item $w'(v,b_k)=t_kw(v,b_0)$ for $k=0,1,\dots,s$.
\end{itemize}
\end{definition}

We can also illustrate $s$-summand duplication at the level of non-symmetric matrices rather than weights, via representatives of equivalence classes. Recall that for any two matrices $A=(a_{i,j})$ and $B$ the {\em Kronecker (or tensor) product of~$A$ and $B$}, denoted by $A \otimes B$, is defined to be $A \otimes B = (a_{i,j}B)$. Keeping the same notation as in Definition~\ref{def:weight-summand-duplication}, choose any $A=\begin{pmatrix}  {a}_{i,j}
\end{pmatrix}\in \cs_w(T)$. Then 
\[ A'=\begin{pmatrix}A[T\setminus B_0]&w(v,b_0){\bf t}^\top\otimes E_{v,b_0}    \\{\bf 1}_{s+1}\otimes E_{b_0,v}&I_{s+1}\otimes A[B_0]\end{pmatrix} \in \cs_{w'}(T').
\]  The matrix $A'=\begin{pmatrix} a'_{i,j}
\end{pmatrix}\in \cs_{w'}(T')$ is called the  \emph{$s$-summand duplication of $A[B_0]$ at $x$ relative to ${\bf t}$}. 
Note:
\begin{gather*}\text{$A'[T\setminus B_0]= A[T\setminus B_0]$, $A'[B_i]=A[B_0]$,}\\ a'_{i,j}=t_k w_A(v,b_0) \text{ if } i<j \text{ and } \{i,j\}=\{v,b_k\}.
\end{gather*}

We now verify that equality~\eqref{eq:branching-spectrum} still holds for this definition.

\begin{proposition}\label{prop:branch-spec}
  Let $v$ be a vertex in a rooted tree $T$, $B_0$ be a branch of~$T$ at $v$, $s\ge 0$, and  ${\bf t}=(t_0,\dots,t_s)\in \bR^{s+1}$ a positive  vector with $t_0+\dots+t_s=1$. If $w' \in W(T')$ is the $s$-summand duplication of $w \in W(T)$ at $v$ for $B_0$ relative to ${\bf t}$,
  then \[\sigma(A')=\sigma(A)\cup s\sigma(A[B_0])\] for any $A \in \cs_w(T)$ and $A' \in \cs_{w'}(T')$.
\end{proposition}

\begin{proof}
  Let $w=w_A$. It suffices to show that the matrices $$A'=\begin{pmatrix}A[T\setminus B_0]&w(v,b_0){\bf t}^\top\otimes E_{v,b_0}   \\{\bf 1}_{s+1}\otimes E_{b_0,v}& I_{s+1}\otimes A[B_0]\end{pmatrix}\in \cs_{w'}(T')$$ and $\tilde A\oplus (I_s\otimes A[B_0])$ are similar, where 
  \[\tilde A:=\begin{pmatrix} A[T\setminus B_0]&w(v,b_0)E_{v,b_0}\\E_{b_0,v}&A[B_0]\end{pmatrix}\in \cs_w(T).\]Let $D$ be the diagonal matrix with diagonal $(\sqrt{t_0},\dots,\sqrt{t_{s}})$. By assumption, $D{\bf 1}_{s+1}$ is a unit vector, so there exists an orthogonal matrix $U$ with $UD{\bf 1}_{s+1}={\bf e}_1$.  Note that $UD^{-1}{\bf t}=UD {\bf 1}_{s+1}={\bf e}_1$.  Let \[
V=I_{|T\setminus B_0|}\oplus (UD\otimes I_{|B_0|}).
\]
Then
\begin{align*}
&VA'V^{-1}\\
&=\begin{pmatrix}I\\& UD\otimes I\end{pmatrix}
\begin{pmatrix}A[T\setminus B_0]&w(v,b_0) {\bf t}^\top\otimes E_{v,b_0}   \\{\bf 1}_{s+1}\otimes E_{b_0,v}&I_{s+1}\otimes A[B_0]\end{pmatrix}
\begin{pmatrix}I\\&D^{-1}U^\top\otimes I   \end{pmatrix}\\
&=\begin{pmatrix}A[T\setminus B_0]&w(v,b_0) {\bf e}_1^\top\otimes E_{v,b_0}   \\{\bf e}_1\otimes E_{b_0,v}&I_{s+1}\otimes A[B_0]\end{pmatrix}\\
&=\begin{pmatrix} A[T\setminus B_0]&w(v,b_0)E_{v,b_0}&{\bf 0}^\top\\E_{b_0,v}&A[B_0]&0\\{\bf 0}&0&I_s\otimes A[B_0]\end{pmatrix}=\tilde{A}\oplus (I_s\otimes A[B_0]).\qedhere
\end{align*}
\end{proof}
Using repeated $s$-summand duplication will be the main constructive technique in this work to produce matrices with high multiplicities of eigenvalues. Conversely, presented with a matrix $A \in \cs(T)$ we will want to know if this matrix was produced by $s$-duplication, and reverse that operation. This reverse process is called ``collapsing''. 

\begin{definition}%
  Let $w$ be a weight function on $T$, where $T$ is a tree, and
let $v$ be a vertex of~$T$. Suppose $B_0,B_1,\dots,B_s$ are mutually isomorphic branches of~$T$ at $v$, with connecting edges $\{v,b_i\}$ for $0\le i\le s$, where the isomorphisms $\theta_i:B_i\to B_0$ have $\theta_i(b_i)=b_0$, and $w|_{B_i}=w|_{B_0}\circ \theta_i$.
Then we say that the branches $B_0,B_1,\dots,B_s$ are \emph{collapsible for $w$}. In this case, we define the \emph{collapsed tree} $T^\circ:=T\setminus(B_1\cup\dots\cup B_s)$ and the \emph{collapsed weight}  $w^\circ\in W(T^\circ)$ which agrees with the restriction of $w$ to $T^{\circ}$, except for 
$$w^\circ(v,b_0):=\sum_{i=0}^s w(v,b_i).$$
\end{definition}

Observe that if $w^\circ\in W(T^\circ)$ is the weight obtained from $w\in W(T)$ by collapsing $B_0,\dots,B_s$, as above, then $w$ may be recovered from $\{w(v,b_i)\}_i$ and $w^\circ$ by $s$-summand duplication of $B_0$ relative to ${\bf t}=\begin{pmatrix}
  \frac{w(v,b_0)}{w^\circ(v,b_0)},\ldots,\frac{w(v,b_s)}{w^\circ(v,b_0)}
\end{pmatrix}\in \bR^{s+1}$. 
By Proposition~\ref{prop:branch-spec},
if $A^\circ\in \ns[w^\circ](T^\circ)$ and $A\in \ns(T)$, then as multisets, we have 
\begin{equation}\label{eq:collapsed-spectrum}
    \text{$s\spec(A[B_0])\subseteq \spec(A)$ and $\spec(A^\circ) = \spec(A)\setminus s\spec(A[B_0])$}. 
    \end{equation}

\begin{definition}
  Let $A\in \cs(T)$, where $T$ is a tree. We say branches $B_0,B_1,\dots,B_s$ of~$T$ (at some common vertex $v$) are \emph{collapsible in $A$} if they are collapsible for the weight $w_A$.
\end{definition}

\begin{definition}\label{def:collapsing-trees-matrices}
  Let $T$ be a tree and $k\ge1$. %
  The tree $T^\circ$ obtained by successively collapsing $B_0,B_1,\dots,B_s$, for each vertex $x\in V(T)$ which is joined to  \pendentpath{k}s $B_0,B_1,\dots,B_s$ in~$T$, is said to be obtained from $T$ by \emph{collapsing \pendentpath{k}s}.
  
Suppose $k\ge1$ and $w\in W(T)$ is a weight with the property that, whenever $B_{i_1}, \ldots, B_{i_t}$ are \pendentpath{k}s in~$T$ which meet a common vertex, as above, the branches $B_{i_1}, \ldots, B_{i_t}$ are collapsible for $w$. We then say that \emph{we can collapse \pendentpath{k}s} in $w$, and the weight $w^\circ\in T^\circ$ obtained by performing all such collapses (for all such collections of \pendentpath{k}s in~$T$) is said to be \emph{obtained from $w$ by collapsing \pendentpath{k}s}. We then say we have \emph{collapsed $\ell$ \pendentpath{k}s}, where $\ell$ is the total number of copies of $P_k$ in $T\setminus T^{\circ}$.
  
  If $A\in \cs(T)$, and we can collapse \pendentpath{k}s in $w_A$, then we say that \emph{we can collapse \pendentpath{k}s in $A$}, and a matrix $A^\circ$ is said to be obtained from $A$ \emph{by collapsing \pendentpath{k}s} if $w_{A^\circ}$ is obtained from $w_A$ by collapsing \pendentpath{k}s.
  
  In the case $k=1$, we refer to these relationships as \emph{collapsing leaves} rather than collapsing \pendentpath{1}s.
\end{definition}

\subsection{The path-to-hedge Construction}
Every hedge $T$ of height $H$ can be constructed from a path $P_{H+1}$ by repeatedly duplicating branches that are pendent paths. In this section, we explain this construction on the level of trees, and then turn to its implications for matrices. 

On the tree level, we start with $P_{H+1}$, and then duplicate $P_{H}$ at the root of the path. 
We then iterate this process, duplicating paths starting at vertices at successively smaller heights. More formally:

\begin{definition}\label{def:chain-subtrees}
  Let $T$ be a hedge of height $H$. For each non-leaf vertex $v$ of~$T$, choose a distinguished child $v^*$ of $v$. We define a  chain of height $H$ hedges
  \[T=T^{(0)}\supseteq T^{(1)}\supseteq\dots\supseteq T^{(H-1)}\supseteq T^{(H)}=P_{H+1} \]
  as follows: for $0\le h\le H$, let $T^{(h)}=T[V^{(h)}]$ be the induced subgraph of~$T$ on the vertex set \[V^{(h)}=(V_H(T)\cup \dots\cup V_{H-h}(T))\cup \{v^*, v^{*(2)},\dots,v^{*(H-h)}: v\in V_{H-h}(T)\},\]
  where $v^{*(k)}:=(v^{*(k-1)})^*$ and $v^{*(1)}:=v^*$.
\end{definition}
The ambiguity in this chain of subgraphs arising from the various possible choices of $v^*$ will not be important, so we fix some arbitrary choice for the rest of this paper, and work with the resulting chain $(T^{(h)})_{0\le h\le H}$. (In examples, we typically choose $v^*$ to be the child of $v$ with the numerically smallest label). Indeed, if we make a different choice of distinguished children, say $\tilde v^*$, yielding a different chain of subgraphs $(\tilde T^{(h)})_{0\le h\le H}$, then the graphs $T^{(h)}$ and $\tilde T^{(h)}$ are clearly isomorphic for each $h$: they are both given by taking the induced subgraph of~$T$ on its vertices of height at least $H-h$, and appending induced pendent paths to each of its leaves to form a hedge of height $H$. The only difference between $T^{(h)}$ and $\tilde T^{(h)}$ is that they are isomorphic induced subgraphs of~$T$ on possibly different vertex sets. 

Observe that for $0\le h<H$, the graph $T^{(h+1)}$ is obtained from $T^{(h)}$ by collapsing \pendentpath{(h+1)}s. Equivalently, $T^{(h)}$ is obtained from $T^{(h+1)}$ by successive branch duplications; more specifically, by $s(v)$-branch duplication of the pendent path at $v$, for every $v\in V_{H-h}(T)$, where $v$ has $s(v)+1$ children in~$T$.
For examples see Figures~\ref{fig:H=2} and~\ref{fig:T31}. Note that on Figure~\ref{fig:T31} we have $T^{(3)}=P_4$, $T^{(2)}$ is the graph spanned by $2P_3\cup P_4$, $T^{(1)}$ is the graph spanned by $6 P_2\cup 2{P_3}\cup 1{P_4}$ and $T^{(0)}=T$. 

\begin{example}\label{ex:BFtreeBranching}
 Consider the two induced subgraphs of the Barioli-Fallat tree $\bftree$ shown in Figures~\ref{fig:P3} and~\ref{fig:T^1},  which are related by branching processes to $\bftree$, shown in Figure~\ref{fig:T2=BFT}.  Indeed, $T^{(1)}$ is obtained from $T^{(2)}$ by $2$-branch duplication of $B=T^{(2)}[\{2,3\}]$ at $v=1$, and $T^{(0)}=\bftree$ is obtained from $T^{(1)}$ by three successive $1$-branch duplications, of the subgraphs of $T^{(1)}$ on $\{3\}$, $\{5\}$ and $\{7\}$, at the vertices $2$, $4$ and $6$, respectively. In terms of collapsing, it is easy to see that for $h=0,1$, the graph $T^{(h+1)}$ is obtained from $T^{(h)}$ by collapsing pendent $(h+1)$-paths.
 
\begin{figure}[htb]
 \tikzset{
    every node/.style={draw, circle, fill=white, inner sep=1pt}
    }
  \begin{subfigure}{0.3\textwidth}
\centering{%
    \begin{tikzpicture}[scale=0.5]
   \node[draw,circle,fill=white] at (0,0) (D) {$1$};
    \node[draw,circle,fill=white] at (-2,-2) (A) {$2$};
    \node[draw,circle,fill=white] at (-2.5,-4) (E) {$3$};
       \draw(D)--(A)--(E);
   \end{tikzpicture}}
    \caption{$T^{(2)}=P_3$}\label{fig:P3}
     \end{subfigure}
     \begin{subfigure}{0.3\textwidth}
\centering{%
     \begin{tikzpicture}[scale=0.5]
    \node[draw,circle,fill=white] at (0,0) (D) {$1$};
    \node[draw,circle,fill=white] at (-2,-2) (A) {$2$};
    \node[draw,circle,fill=white] at (-2.5,-4) (E) {$3$};	
    \node[draw,circle,fill=white] at (0,-2) (B) {$4$};
    \node[draw,circle,fill=white] at (-0.5,-4) (G) {$5$};
    \node[draw,circle,fill=white] at (2,-2) (C) {$6$};
    \node[draw,circle,fill=white] at (1.5,-4) (I) {$7$};
       \draw(D)--(A)--(E);
      \draw(D)--(B)--(G); 
      \draw(D)--(C)--(I);
    \end{tikzpicture}}
    \caption{$T^{(1)}$}\label{fig:T^1}
    \end{subfigure}
    \begin{subfigure}{0.35\textwidth}\centering{\begin{tikzpicture}[scale=0.5]
   \node[draw,circle,fill=white] at (0,0) (D) {$1$};
    \node[draw,circle,fill=white] at (-2,-2) (A) {$2$};
    \node[draw,circle,fill=white] at (-2.5,-4) (E) {$3$};
    \node[draw,circle,fill=white] at (0,-2) (B) {$4$};
    \node[draw,circle,fill=white] at (-0.5,-4) (G) {$5$};
    \node[draw,circle,fill=white] at (2,-2) (C) {$6$};
    \node[draw,circle,fill=white] at (1.5,-4) (I) {$7$};
    \node[draw,circle,fill=white] at (-1.5,-4) (F) {$8$};
    \node[draw,circle,fill=white] at (0.5,-4) (H) {$9$};
    \node[draw,circle,fill=white] at (2.5,-4) (J) {\footnotesize$10$};
        \draw(D)--(A)--(E);
        \draw(B)--(G); 
        \draw(C)--(I);
        \draw (H)--(B)--(D)--(C)--(J);
        \draw (A)--(F);
    \end{tikzpicture}}\caption{$T^{(0)}=\bftree$}\label{fig:T2=BFT}\end{subfigure}
    \caption{Subgraphs $T^{(2)}$ and $T^{(1)}$ of the Barioli-Fallat tree $T^{(0)}=\bftree$.}\label{fig:BFtreeInduced}
\end{figure}
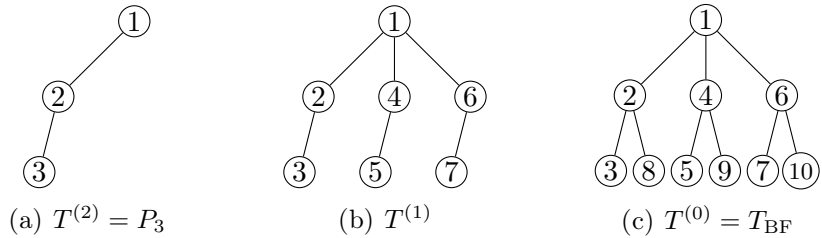    %
\end{example}

We now turn to matrices, starting with certain matrices in $\cs(P_{H+1})$. Given real numbers $a_k$ for $k\ge1$ and $b_k>0$ for $k\ge2$, let $C_n$ be the $n\times n$ matrix defined recursively by
\begin{equation}\label{eq:recursiveM}
 C_1:=(a_1)\quad\text{and}\quad C_{k+1}:=
  \begin{pmatrix}
    a_{k+1}&b_{k+1}{\bf e}_1\\{\bf e }_1^\top&C_{k}
  \end{pmatrix},\; k\ge 1
  \end{equation}
where ${\bf e}_1:=(1,0,0,\dots,0)\in \bR^k$. Hence, $C_n$ is a tridiagonal matrix with main diagonal $(a_n,a_{n-1},\dots,a_1)$, positive super-diagonal $(b_n,b_{n-1},\dots,b_2)$,
and all sub-diagonal elements equal to $1$. For example,
$$C_3=\left(
\begin{array}{ccc}
 a_3 & b_3 & 0   \\
 1   & a_2 & b_2 \\
 0   & 1   & a_1 \\
\end{array}
\right).$$

As we will follow our process with a small running example, let us choose concrete $C_i$, $i=1,2,3$, that will be used in the example.

\begin{example}\label{ex:Ms}
 Let $a_1=2$, $a_2=4$, $a_3=8$, $b_2=3$ and $b_3=20$. Then
 \[C_1=(2), \quad C_2=\left(
\begin{array}{cc}
 4 & 3 \\
 1 & 2 \\
\end{array}
\right), \quad C_3=\left(
\begin{array}{ccc}
 8 & 20 & 0 \\
 1 & 4 & 3 \\
 0 & 1 & 2 \\
\end{array}
\right).\]
We remark for future use that $\sigma(C_2)=\{1,5\}$ and $\sigma(C_3)=\{0,3,11\}$.
\end{example}

\begin{definition}
  Let $T$ be a hedge of height~$H$ and $w\in W(T)$. Let $w'\in W(P_{H+1})$, where, as usual, $P_{H+1}$ is viewed as a hedge of height~$H$. For $v\in V(T)$, let us write $v'$ for the unique vertex of $P_{H+1}$ with the same height as $v$.  We say that the weight \emph{$w$ comes from $w'$ via the path-to-hedge construction} if:
  \begin{itemize}
      \item $w(v)=w'(v')$ for all $v\in V(T)$; and
      \item for every non-leaf vertex $v$ of~$T$, we have
      \[\sum_{u\in \children(v)} w(v,u)=w'(v',z)\]
      where $z$ is the child of $v'$ in $P_{H+1}$.
  \end{itemize}
  We write $\PHset(w',T)$ for the set of all such weights~$w$. %
  
  If $C\in \ns[w'](P_{H+1})$, $A\in \cs(T)$, and $w_A\in \PHset(w',T)$, then we say that \emph{$A$ comes from $C$ via the path-to-hedge construction}, and we write $\PHset(C,T)$ for the set of all such matrices $A$.
\end{definition}

\begin{proposition}\label{prop:collapsing}
Let $T$ be a hedge of height~$H$, and recall the sequence of hedges $T=T^{(0)}\supseteq\dots\supseteq T^{(H)}=P_{H+1}$ from Definition~\ref{def:chain-subtrees}. Suppose $A^{(h)}\in \cs(T^{(h)})$
for $0\le h\le H$, and that for $h\in[H]$, the matrix $A^{(h)}$ is obtained by collapsing \pendentpath{h}s in $A^{(h-1)}$. Then $A^{(0)}\in \PHset(A^{(H)},T)$. 

Moreover, for $h\in [H]$, we have $w_{A^{(h-1)}}|_B=w_{A^{(H)}}|_{B'}$ whenever $B$ is a \pendentpath{h} in $T^{(h-1)}$ and $Q$ is the \pendentpath{h} in $P_{H+1}$.
\end{proposition}
\begin{proof}
  We claim that $A^{(h)}\in \PHset(A^{(H)},T^{(h)})$ for $0\le h\le H$. This is trivial for $h=H$, and follows inductively for $0\le h<H$. %
  
  Let $h\in [H]$. Since $A^{(h)}$ is obtained by collapsing \pendentpath{h}s in $A^{(h-1)}$, we have $w_{A^{(h-1)}}|_B=w_{A^{(h)}}|_{B'}$ for all \pendentpath{k}s $B$ in $T^{(h-1)}$ and $B'$ in $T^{(h)}$, whenever $1\le k\le h$. By induction, $w_{A^{(h-1)}}|_B=w_{A^{(H)}}|_Q$ where $Q$ is the \pendentpath{k} in $P_{H+1}$, and taking $k=h$ establishes the claim.
\end{proof}

The eigenvalues of $A \in \PHset(C,T)$ are determined by the eigenvalues of $C$ and a selection of principal submatrices of $C$, as follows.

\begin{theorem}\label{thm:arbitrary M}
  Let $T$ be a hedge of height~$H$, $C \in \cs(P_{H+1})$ and $A \in \PHset(C,T)$. For $i=1,\ldots,H+1$ define $C_i:=C[H-i+2,\ldots ,H+1]$. %
  Then
  \[\spec(A)=\bigcup_{i=1}^{H+1} \ell_i(T)\spec(C_i).\]
\end{theorem}

\begin{proof}%
Let $w^{(0)}=w_A$. Since $A^{(0)}\in \PHset(C,T^{(0)})$, we can collapse $\ell_1(T)$ \pendentpath{1}s in $w^{(0)}$ to obtain a weight $w^{(1)}\in \PHset(w_C, T^{(1)})$. Continuing inductively (collapsing $\ell_{h}(T)$ \pendentpath{h}s at the $h$th step), we obtain $w^{(h+1)}\in \PHset(w_C,T^{(h)})$ for $0\le h< H$. In particular, $w^{(H)}\in \PHset(w_C,T^{(H)})=\ns[w_C](P_{H+1})$, so $w^{(H)}=w_C$. Now choose a matrix $A^{(h)}\in \ns[w^{(h)}](T^{(h)})$ for each $h\in [H]$. Since $w^{(h)}$ is obtained from $w^{(h-1)}$ by collapsing $\ell_h(T)$ branches $B$, each with $w^{(h-1)}|_B=w_{C_h}$ by Proposition~\ref{prop:collapsing},
equation~\eqref{eq:branching-spectrum} implies that for $h\in [H]$ we have \[\sigma(A^{(h-1)})=\ell_h(T)\sigma(C_h)\cup \sigma(A^{(h)}).\]
Moreover, $\sigma(A^{(H)})=\sigma(C)=\ell_{H+1}(T)\sigma(C_{H+1})$. Hence, $\sigma(A)=\sigma(A^{(0)})=\bigcup_{i=1}^{H+1}
\ell_i(T)\sigma(C_i)$, as claimed.
\end{proof}

\begin{definition}
  Let $T$ be a hedge of height~$H$, $C \in \cs(P_{H+1})$ and $A \in \PHset(C,T)$. For $i=1,\ldots,H+1$ define $C_i:=C[H-i+2,\ldots ,H+1]$. We refer to $\sigma(C_i)$ as the set of \emph{level~$i$ eigenvalues of~$A$}, for $1\le i\le H+1$.
\end{definition}

\begin{example}\label{eg:a1-bftree}
Let $C_3=\left(\begin{smallmatrix}8&20&0\\1&4&3\\0&1&2\end{smallmatrix}\right)$ be the matrix of Example~\ref{ex:Ms}. Then $\PHset(C_3,\bftree)$ is the set of matrices of the form
\begin{equation}\label{eq:C-BF-example} 
A=
\left(
\begin{smallmatrix}
8 & 20 t_{1,2} & 0              & 20 t_{1,4} & 0              & 20 t_{1,6} & 0              & 0              & 0              & 0               \\
 1       & 4        & 3 t_{2,3} & 0              & 0              & 0              & 0              & 3 t_{2,8} & 0              & 0               \\
 0       & 1              & 2        & 0              & 0              & 0              & 0              & 0              & 0              & 0               \\
 1       & 0              & 0              & 4        & 3 t_{4,5} & 0              & 0              & 0              &  3 t_{4,9} & 0               \\
 0       & 0              & 0              & 1              & 2        & 0              & 0              & 0              & 0              & 0               \\
 1       & 0              & 0              & 0              & 0              & 4        & 3 t_{6,7} & 0              & 0              & 3 t_{6,10} \\
 0       & 0              & 0              & 0              & 0              & 1              & 2        & 0              & 0              & 0               \\
 0       & 1              & 0              & 0              & 0              & 0              & 0              & 2        & 0              & 0               \\
 0       & 0              & 0              & 1              & 0              & 0              & 0              & 0              & 2        & 0               \\
 0       & 0              & 0              & 0              & 0              & 1              & 0              & 0              & 0              & 2         \\
\end{smallmatrix}
\right),
\end{equation}
where $t_{v,b}>0$ and  
$$\sum_{b\in \children(v)} t_{v,b}=1$$
for $v\in \{1,2,4,6\}$.
 By Theorem~\ref{thm:arbitrary M}, all these matrices are co-spectral, and since $(\ell_3(\bftree),\ell_2(\bftree),\ell_1(\bftree))=(1,2,3)$, we have
 \begin{align*}
     \sigma(A)&=\sigma(C_3)\cup 2\sigma(C_2)\cup 3\sigma(C_1)\\
     &=\{0,3,11\}\cup 2\{1,5\}\cup 3\{2\}=\{0,1,1,2,2,2,3,5,5,11\}
 \end{align*}
 as presented in Table~\ref{tab:C-eigenvalues}.
  \begin{table}[htb!]
    \begin{center}
    \begin{tabular}{c|cccccc|l}
        matrix &\multicolumn{6}{c|}{eigenvalues} & $\ell_i$s\\
        \hline
        $C_3$ & $0$&&&$3$&&$11$&$\ell_3=1$\\
        $C_2$ & &$1$&&&$5$&&$\ell_2=2$\\
        $C_1$ & &&$2$&&&&$\ell_1=3$\\
        \hline 
        $A$ & $0$ & $1^{(2)}$& $2^{(3)}$ & $3$&  $5^{(2)}$ & $11$
    \end{tabular}
\end{center}
 \caption{The eigenvalues of $A \in \PHset(C_3,\bftree)$ and its level~$i$ eigenvalues (i.e., the eigenvalues of $C_i$) for $i=1,2,3$.}
     \label{tab:C-eigenvalues}
 \end{table}
\end{example}

Theorem~\ref{thm:arbitrary M} shows that, starting with any matrix $C \in \cs(P_{H+1})$, the path-to-hedge construction produces a matrix with several multiple eigenvalues. Generically (when $\sigma(C_i)\cap \sigma(C_j)=\emptyset$ for $i\ne j$ as in Example \ref{eg:a1-bftree}), we will only obtain the multiplicities $\ell_i(T)$, as implied by the theorem. However, it is possible to increase the multiplicities of selected eigenvalues even further, by carefully choosing matrices $C \in \cs(P_{H+1})$ so that appropriate submatrices $C_i$ and $C_j$ of $C$ have some eigenvalues in common; the multiplicities of such eigenvalues in a matrix $A\in \PHset(C,T)$ will then be at least $\ell_i+\ell_j$. We will present one systematic way to do this in the next section; for now we give a small example.

\begin{example}\label{ex:coincidences}
Consider the construction in Example \ref{eg:a1-bftree} that uses $$C'_3=\left(
\begin{array}{ccc}
 2 & 20 & 0 \\
 1 & 4 & 3 \\
 0 & 1 & 2 \\
\end{array}
\right)$$ 
with eigenvalues $\left\{3-2 \sqrt{6},2,3+2 \sqrt{6}\right\}$ instead of $C_3$ defined in Example \ref{ex:Ms}.  Then the resulting matrices $A'\in \PHset(C_3',\bftree)$ have $(1,1)$ element equal to $2$, and are otherwise identical to the matrices in \eqref{eq:C-BF-example}. Since $2$ is now a common eigenvalue of $C'_3$ and $C'_1$, it has multiplicity $4$ as an eigenvalue of~$A'$.  Indeed,
 \begin{align*}
     \sigma(A')&=\sigma(C'_3)\cup 2\sigma(C'_2)\cup 3\sigma(C'_1)\\
     &=\{3-2\sqrt 6,2,3+2\sqrt6\}\cup 2\{1,5\}\cup 3\{2\}\\&
     =\{3-2 \sqrt{6},1,1,2,2,2,2,5,5,3+2 \sqrt{6}\},
 \end{align*}
 as presented in Table~\ref{tab:bftree}. Note that $A'$ realizes the multiplicity list ${\ordm}_{BF}=(1,2,4,2,1)$ of  Example~\ref{ex:BF-12421}.
  \begin{table}[htb!]
    \begin{center}
    \begin{tabular}{c|ccccc|l}
        matrix &\multicolumn{5}{c}{eigenvalues} & $\ell_i$s\\
        \hline
        $C'_3$ & $3-2\sqrt{6}$&&$2$&&$3+2\sqrt{6}$&$\ell_3=1$\\
        $C'_2$ & &$1$&&$5$&&$\ell_2=2$\\
        $C'_1$ & &&$2$&&&$\ell_1=3$\\
        \hline 
        $A'$ & $3-2\sqrt{6}$ & $1^{(2)}$& $2^{(4)}$ &  $5^{(2)}$ & $3+2\sqrt{6}$
    \end{tabular}
\end{center}
 \caption{The eigenvalues of $A' \in \PHset(C'_3,\bftree)$ and its level~$i$ eigenvalues (i.e., the eigenvalues of $C'_i$) for $i=1,2,3$.}
     \label{tab:bftree}
 \end{table}
\end{example}

\section{A greedy construction of a path matrix}\label{sec:path-construction}

In the path-to-hedge construction, we now have a method to construct a matrix corresponding to a hedge whose eigenvalues coincide with the eigenvalues of the trailing principal submatrices of a given path matrix~$C$. We now aim to choose $C$ so that these trailing submatrices have common and specified eigenvalues, occuring periodically. 
This is done in Proposition~\ref{prop:interlacing},  starting from five distinct specified eigenvalues satisfying certain ordering conditions~\eqref{ineq:regions}. In Theorem~\ref{thm:section-4-highlights}, we apply the path-to-hedge construction to obtain a matrix for a given lush hedge that has these five specified eigenvalues with high multiplicities. 
It is interesting that these five specified eigenvalues need not be the only eigenvalues of the matrix with high multiplicities: in Section~\ref{sec:rigid} we will show that as the height of the lush hedge increases, the matrix constructed in this section may achieve other eigenvalues with high multiplicities. 

\subsection{Eigenvalue coincidences in subpaths}
Let $a_n\in \bR$, $b_n>0$ and 
$C_n$ be defined as in~\eqref{eq:recursiveM}. To increase the multiplicities of matrices arising from the path-to-hedge construction even further than those that arise automatically from Theorem~\ref{thm:arbitrary M}, we want to choose $a_n,b_n$ so that for certain $i,j$, the matrices $C_i$ and $C_j$ have some eigenvalues in common. As we will shortly see, $C_i$ and $C_{i+1}$ cannot have any eigenvalues in common, so, we first demand that $C_{i}$ and $C_{i+2}$ always have a common eigenvalue. In addition, we can have a common eigenvalue for $C_i$ and $C_{i+3}$.  

\begin{lemma}\label{lem:steps}
For $n\ge1$, let $C_n\in \cs(P_n)$ be defined as in~\eqref{eq:recursiveM}, where $a_i\in \bR$ for $i\ge1$ and $b_i>0$ for $i\ge2$. 
\begin{enumerate}
\item If $k\ge2$, then $\sigma(C_{k-1})\cap \sigma(C_k)=\emptyset$.
\item If $k\ge3$ and  $\alpha \in \spec(C_{k-2})$, then $\alpha \in \spec(C_k)$ if and only if $a_k=\alpha$. 
\item If $k\ge4$ and $\beta \in \spec(C_{k-3})$, then $\beta \in \spec(C_k)$ if and only if $b_k=(\beta-a_k)(\beta-a_{k-1})$. 
\end{enumerate}
\end{lemma}
\begin{proof}
Recall that there is a symmetric matrix $A\in \S(P_k)$ which is cospectral with $C_k$; then $A(1)$ is cospectral with $C_{k-1}$. The eigenvalues of~$A$ and $A(1)$ strictly interlace (see for example~\cite[Theorem~4.3.17]{horn_johnson_2012}), hence the first claim holds.

Let $p_k(x)$ denote the characteristic polynomial of $C_k$. Then, for $k\ge3$,
$$p_k(x)=(x-a_{k})p_{k-1}(x)-b_k p_{k-2}(x).$$
In particular, assuming $p_{k-2}(\alpha)=0$, we get $p_{k-1}(\alpha)\neq 0$, by our initial observation, and $p_k(\alpha)=(\alpha-a_k)p_{k-1}(\alpha)$. The second part of the lemma follows. 

To prove the final part, suppose $k\ge4$ and $p_{k-3}(\beta)=0$. From the recursive relation above we get:
\begin{equation*}\label{eq:recursive_two}
p_k(x)=((x-a_k)(x-a_{k-1})-b_k)p_{k-2}(x)-(x-a_k)b_{k-1}p_{k-3}(x),
\end{equation*}
so $p_k(\beta)=((\beta-a_k)(\beta-a_{k-1})-b_k)p_{k-2}(\beta)$. 
Since $p_{k-2}(\beta)\neq 0$, we get $p_k(\beta)=0$ if and only if  $b_k=(\beta-a_k)(\beta-a_{k-1})$, as required.
\end{proof}

We now give some conditions whose relevance will become apparent in Proposition~\ref{prop:interlacing} below.
\begin{definition}
Let $\mathcal B$ denote the set of all $\Lambda=(\alpha_1,\alpha_2,\beta_2,\beta_3,\beta_4)\in \bR^5$ with five distinct elements, that satisfy one of following twelve conditions:
\begin{equation}\label{ineq:regions}\left\{\quad
\begin{aligned}
 \beta _2<\alpha _1<\alpha _2<\beta _3<\beta _4 &  \\
 \beta _2<\beta _4<\alpha _1<\alpha _2<\beta _3 &  \\
 \beta _4<\beta _2<\alpha _1<\alpha _2<\beta _3 & \text{ and }\alpha _2+\beta _2>\beta _4+\beta _3  \\
 \beta _3<\beta _2<\alpha _1<\alpha _2<\beta _4 & \text{ and }\alpha _2+\beta _2<\beta _4+\beta _3  \\
 \beta _3<\beta _2<\beta _4<\alpha _1<\alpha _2 &  \\
 \beta _4<\beta _3<\beta _2<\alpha _1<\alpha _2 & \\[6pt]
 \beta _4<\beta _3<\alpha _2<\alpha _1<\beta _2 &   \\
 \beta _3<\alpha _2<\alpha _1<\beta _4<\beta _2 &   \\
 \beta _3<\alpha _2<\alpha _1<\beta _2<\beta _4 & \text{ and }\alpha _2+\beta _2<\beta _4+\beta _3  \\
 \beta _4<\alpha _2<\alpha _1<\beta _2<\beta _3 & \text{ and }\alpha _2+\beta _2>\beta _4+\beta _3  \\
 \alpha _2<\alpha _1<\beta _4<\beta _2<\beta _3 &  \\
\alpha _2<\alpha _1<\beta _2<\beta _3<\beta _4 &.
\end{aligned}\right.
\end{equation}
We also define
\begin{align*} \mathcal{B}_3&=\{(\alpha_1,\alpha_2,\beta_2,\beta_3)\in \bR^4 : (\alpha_1,\alpha_2,\beta_2,\beta_3,\beta_4)\in \mathcal{B}\text{ for some $\beta_4\in \bR$}\},\\
\mathcal{B}_2&=\{(\alpha_1,\alpha_2,\beta_2)\in \bR^3 : (\alpha_1,\alpha_2,\beta_2,\beta_3,\beta_4)\in \mathcal{B}\text{ for some $\beta_3,\beta_4\in \bR$}\}
\end{align*}
and set $\mathcal{B}_1=\bR$. (Our convention here is that $\mathcal{B}_i$ is the set $\mathcal{B}$ with any variables with a subscript $j>i$ deleted.)
\end{definition}

Note that the latter six conditions in~\eqref{ineq:regions} are the reverse of the former six, so that
up to negation the number of disjoint valid regions in~$\mathcal{B}$ is six. (For $\mathcal{B}_3$ and $\mathcal{B}_2$, these collapse to two regions and one valid region up to negation, respectively.)
Note also that under the linear substitution $\gamma = \alpha_2 + \beta_2 - \beta_4$,
the third condition is equivalent to a single string of inequalities
\[
\beta_2 < \alpha_1 < \alpha_2 < \beta_3 < \gamma
\]
and the fourth condition is similarly equivalent to
\[
\gamma < \beta_3 < \beta_2 < \alpha_1 < \alpha_2,
\]
from which we can recover, in either case, $\beta_4 = \alpha_2 + \beta_2 - \gamma$.
It follows that the $12$ disjoint conditions can each, separately, be
parameterized by choosing five real numbers that are strictly ordered but otherwise completely free.
\bigskip

The following ``periodic'' notation will turn out to be convenient. 
\begin{definition}\label{def:alpha-beta}
Given real numbers $\alpha_1,\alpha_2,\beta_2,\beta_3,\beta_4$, we extend these to sequences $(\alpha_i)$, $(\beta_i)$ as follows:%
\[ \alpha_i:=\alpha_k\quad\text{whenever $i\not\in \{1,2\}$, $k\in \{1,2\}$ and $i\in 2\bZ+k$}\]
and
\[ \beta_i:=\beta_k\quad\text{whenever $i\not\in \{2,3,4\}$, $k\in \{2,3,4\}$ and $i\in 3\bZ+k$.}\]
\end{definition}

\begin{proposition} \label{prop:interlacing}
Let $\Lambda :=  (\alpha_1,\alpha_2,\beta_2,\beta_3,\beta_4)$ be a list of five distinct real numbers satisfying %
$\alpha_2+\beta_2\ne \beta_3+\beta_4$. Let  $n\geq 1$ be an integer, $a_i$ and $b_i>0$ be real numbers %
and consider the corresponding matrices $C_i \in \R(P_i)$, $i=1,\ldots,n$, given by \eqref{eq:recursiveM}.
  The following are equivalent:
   \begin{enumerate}[(i)]
   \item \label{cond:i}
   $\spec(C_1)=\{\alpha_1\}$ and $\spec(C_i)\cap \Lambda=\{\alpha_i,\beta_i\}$ for $i=2,\ldots,n$. 
   \item\label{cond:ii} The entries $a_i,b_i$ of the matrices $C_1,\dots,C_n$ are given by 
   \begin{equation}\label{eq:a_is}
  a_i=
  \begin{cases}
    \alpha_1                   & i\in 2\bZ+1, \\
    -\alpha_1+\alpha_2+\beta_2 & i=2,         \\
    \alpha_2                   & 2<i\in 2\bZ,
  \end{cases}
  \end{equation}
  for $1\le i \le n$, and 
  \begin{equation}\label{eq:b_is}
  b_i=
  \begin{cases}
    ( \beta_2-\alpha_1)(\alpha_1-\alpha_2) & i=2,                                     \\
    (\beta_3-\alpha_2)(\beta_3-\beta_2)   & i=3,                                     \\
    \frac{(\beta_4-\alpha_1)(\beta_3-\beta_4)(\alpha_2+\beta_2-\beta_3-\beta_4)}{\beta_4-\beta_2}
                                          & i=4,                                     \\
    (\beta_j-\alpha_1)(\beta_j-\alpha_2)  & j\in \{2,3,4\}\text{ and }4<i\in 3\bZ+j, \\
  \end{cases}
\end{equation} for $2\le i\le n$.
\end{enumerate}
Moreover, when these conditions are satisfied, the following assertions hold.  If $n\ge 4$, then $\Lambda\in \mathcal{B}$. If  $n=3$, then $(\alpha_1,\alpha_2,\beta_2,\beta_3)\in \mathcal{B}_3$. If $n=2$, then $(\alpha_1,\alpha_2,\beta_2)\in \mathcal{B}_2$. If $n=1$, then $\alpha_1\in \mathcal{B}_1$.
\end{proposition}
\begin{proof} 
 For $n\le 2$, this is simple to verify directly. 

In the case $n=3$, suppose $\{\alpha_1, \beta_3\} \subseteq \spec(C_3)$. Since 
$\sigma(C_1)=\{\alpha_1\}$,
Lemma~\ref{lem:steps} allows us to conclude that $a_3 = \alpha_1$. Now, the characteristic polynomial $p_3(x)$ of $C_3$ is 
$p_3(x)=(x-\alpha_1)(x-\alpha_2)(x-\beta_2)-b_3(x-\alpha_1)$. Since $p_3(\beta_3)=0$, this implies that $b_3=(\beta_3-\alpha_2)(\beta_3-\beta_2)$, as claimed.  Conversely, given $a_i,b_i$ as above, we can check directly that $\spec(C_3)=\{\alpha_1,\beta_3,\alpha_2+\beta_2-\beta_3\}$. The eigenvalue $\alpha_2+\beta_2-\beta_3$ is not equal to $\beta_4$ by hypothesis, and it is not equal to $\beta_2\in \sigma(C_2)$, by Lemma~\ref{lem:steps}. So $\spec(C_3)\cap \Lambda=\{\alpha_3,\beta_3\}$, as required. (We note here for future use that in the case $n=3$, if, contrary to our hypothesis, we do have $\alpha_2+\beta_2= \beta_3+\beta_4$, then condition (2) is equivalent to the conditions $\sigma(C_1)=\{\alpha_1\}$, $\sigma(C_2)=\{\alpha_2,\beta_2\}$ and $\sigma(C_3)=\{\alpha_1,\beta_3,\beta_4\}$.)  Moreover, since %
$b_2,b_3>0$ and (by the $n=2$ case) we have $b_2=(\beta_2-\alpha_1)(\alpha_1-\alpha_2)$, it is straightforward to see that  $(\alpha_1,\alpha_2,\beta_2,\beta_3) \in \mathcal{B}_3$.

For $n=4$, supposing that $\{\alpha_2,\beta_4\}\subseteq \sigma(C_4)$, a similar argument as was used for $n=3$ gives $a_4=\alpha_2$ and $b_4=\frac{p_3(\beta_4)}{\beta_4-\beta_2}$, as claimed in condition~\ref{cond:ii}. The requirement $b_i>0$ for $i=2,3,4$ then implies that $\Lambda=(\alpha_1,\alpha_2,\beta_2,\beta_3, \beta_4) \in \mathcal{B}$ whenever  $n\ge4$ and condition~\ref{cond:ii} holds.

For $n\ge 5$, suppose we have the given spectral conditions~\ref{cond:i} on $C_i$ for $1\le i\le n$. In particular, $\alpha_n=\alpha_{n-2}\in \sigma(C_{n-2})$ and $\beta_n=\beta_{n-3}\in \sigma(C_{n-3})$. 
The given formulae for $a_n$ and $b_n$ then follow directly from Lemma~\ref{lem:steps}.

Conversely, suppose $n\ge 4$ and $a_i,b_i$ are given by \eqref{eq:a_is} and~\eqref{eq:b_is} for $i\ne n+1$. Considering the matrices $C_{n-3}$ and $C_{n-2}$, Lemma~\ref{lem:steps} shows that $\{\alpha_n,\beta_n\}\in \sigma(C_n)$. Moreover,
by Lemma~\ref{lem:steps} again, $\spec(C_{n})$ does not intersect $\spec(C_{n-1})\cup \spec(C_{n+1})\supseteq \{\alpha_{n-1},\beta_{n-1},\beta_{n+1}\}$, so the intersection of $\sigma(C_n)$ with $\Lambda$ is precisely $\{\alpha_n,\beta_n\}$.
  It is routine to check that if $n\ge4$, then $b_i>0$ for $2\le i\le n$ implies that exactly one of the $12$ disjoint conditions in~\eqref{ineq:regions} is satisfied.
\end{proof}

\begin{definition}\label{def:B}
Given $\Lambda \in \mathcal B$, we consistently write %
$\Lambda=(\alpha_1,\alpha_2,\beta_2,\beta_3,\beta_4)$ and
    \begin{equation*}\label{eq:C-n-Lambda}
        C_n^{\Lambda}=\begin{pmatrix}
      a_n&b_n\\
      1&a_{n-1}&b_{n-1}\\
      &1&a_{n-2}&b_{n-2}\\
      &&\ddots&\ddots&\ddots\\
      &&&1&a_2&b_2\\
      &&&&1&a_1&\\
    \end{pmatrix}   
    \end{equation*}
    for the unique $n\times n$ matrix satisfying Proposition~\ref{prop:interlacing}. %
\end{definition}

Explicitly, we have
\bgroup\allowdisplaybreaks
\begin{align*}
  C_1^{\Lambda}&=\left(
\begin{array}{c}
 \alpha_1 \\
\end{array}
\right),\\
C_2^{\Lambda}&=\left(
\begin{array}{cc}
 -\alpha_1+\alpha_2+\beta_2 & (\alpha_1-\alpha_2) (\beta_2-\alpha_1) \\
 1                          & \alpha_1                               \\
\end{array}
\right),\\C_3^{\Lambda}&=\left(
\begin{array}{ccc}
 \alpha_1 & (\alpha_2-\beta_3) (\beta_2-\beta_3) & 0                                      \\
 1        & -\alpha_1+\alpha_2+\beta_2           & (\alpha_1-\alpha_2) (\beta_2-\alpha_1) \\
 0        & 1                                    & \alpha_1                               \\
\end{array}
\right),\\
C_4^{\Lambda}&=\left(
\begin{array}{cccc}
 \alpha_2 & b_4 & 0                                    & 0                                      \\
 1        & \alpha_1                                                                                        & (\alpha_2-\beta_3) (\beta_2-\beta_3) & 0                                      \\
 0        & 1                                                                                               & -\alpha_1+\alpha_2+\beta_2           & (\alpha_1-\alpha_2) (\beta_2-\alpha_1) \\
 0        & 0                                                                                               & 1                                    & \alpha_1                               \\
\end{array}
\right)
\intertext{with $b_4 = \tfrac{(\alpha_1-\beta_4) (\beta_4-\beta_3) (\alpha_2+\beta_2-\beta_3-\beta_4)}{\beta_4-\beta_2}$  as in \eqref{eq:b_is}, and, for $n > 4$,}
C_n^{\Lambda}&=\left(
\begin{array}{c|c}
\alpha_n & \begin{array}{cc}(\beta_n - \alpha_1)(\beta_n - \alpha_2) & {\mathbf 0}^\top \end{array} \\\hline
    \begin{array}{c}
        1 \\
        \mathbf{0}
    \end{array} & C_{n - 1}^{\Lambda}
\end{array}
\right).
\end{align*}
\egroup

Note that for $n>4$, these matrices have periodic diagonal and periodic superdiagonal, if we remove the last $3$ rows and columns. In fact, the diagonal of $C_n^{\Lambda}$ with its final two entries omitted is of the form $(\alpha_n,\alpha_{n-1},\dots)$, and hence has period $2$, and the superdiagonal of $C_n^{\Lambda}$ with its final $3$ entries omitted is of the form $((\beta_n-\alpha_1)(\beta_n-\alpha_2),(\beta_{n-1}-\alpha_1)(\beta_{n-1}-\alpha_2),\dots)$, and hence has period $3$.

\subsection{Consequences for the path-to-hedge construction}
\begin{corollary}\label{cor:critical_mult}
 If $T$ is a hedge of height $H$, $\Lambda=(\alpha_1,\alpha_2,\beta_2,\beta_3,\beta_4) \in \mathcal B$, and $A \in \PHset(C_{H+1}^{\Lambda},T)$, then
 \begin{align*}
     \mult(\alpha_j,A)&=\sum\limits_{\substack{i\in 2\bZ+j,\\ i\ge j}}
    \ell_i(T),\; j=1,2\quad\text{and}\\
   \mult(\beta_j,A)&=\sum\limits_{\substack{i\in 3\bZ+j,\\i\ge j}}
    \ell_i(T),\; j=2,3,4.
 \end{align*}
\end{corollary}

\begin{proof}
  Theorem \ref{thm:arbitrary M}
  tells us that 
  \[\spec(A)=\bigcup_{i=1}^{H+1} \ell_i(T)\spec(C_i^{\Lambda}).\]
 From Proposition \ref{prop:interlacing} we get 
  $\spec(C_1^{\Lambda})=\{\alpha_1\}$ and $\spec(C_i^{\Lambda})\cap \Lambda=\{\alpha_i,\beta_i\}$ for $i=2\ldots,n$. The claim follows. 
\end{proof}
  
\begin{definition}
  \label{def:critical-multiplicity2}
  If $T$ is a lush hedge of height $H \ge 2$, then we say that an unordered multiplicity list $\ordm$ is a \emph{\highMultiplicityList{}} for $T$
  if there exist five multiplicities $m_1, m_2, n_2, n_3, n_4$ in $\ordm$ so that:
  \begin{equation}\left\{\quad
  \begin{aligned}\label{ineqs:high-mults2}
    m_j &\ge \sum\limits_{\substack{i\in 2\bZ+j,\\ i\ge j}}
    \ell_i(T),\; j=1,2; \\
     n_j&\ge \sum\limits_{\substack{i\in 3\bZ+j,\\i\ge j}}
    \ell_i(T),\; j=2,3,4; \quad\text{and} \\
    n_4 &< \ell_3(T).
  \end{aligned}\right.
  \end{equation}
\end{definition}
Note that in the case $H=2$, for which $\ell_3 = 1$, this requires $\ordm$
to include a multiplicity $n_4 = 0$.
This is possible if we adopt the convention that a spectrum, and correspondingly
an ordered or unordered multiplicity list, may be written down so as to include an ``eigenvalue of multiplicity zero''
that does not actually occur in the spectrum.
The inequalities in \eqref{ineqs:high-mults2} are derived from the multiplicities of entries of~$\Lambda$ as eigenvalues of $A \in \PHset(C_{H+1}^{\Lambda},T)$. Since $\spec(C_3^{\Lambda})$ contains an eigenvalue $\gamma \not\in \Lambda$, $\gamma$ has multiplicity  at least $\ell_3$ in $A$. Hence, for trees of small height the multiplicities of entries in $\Lambda$ need not be the highest five multiplicities of~$A$, and that is why the inequality $n_4< \ell_3(T)$ is needed above. %

Elements of $\Lambda \in \mathcal B$ contain distinguished eigenvalues of matrices $C_i^{\Lambda}$, and hence also of matrices in $\PHset(C_{i}^{\Lambda},T)$. %

We summarise the outcomes of this subsection in the theorem below. 

\begin{theorem}\label{thm:section-4-highlights}
Let $T$ be a lush hedge of height $H\ge2$ and $\Lambda \in \mathcal B$. Every matrix $A \in \PHset(C_{H+1}^{\Lambda},T)$ has a critical multiplicity list.
\end{theorem}

\begin{proof}
    For any $\Lambda \in \mathcal B$ and $A \in \PHset(C_{H+1}^{\Lambda},T)$ let us denote $m_i=\mult(\alpha_i,A)$ and  $n_j=\mult(\beta_j,A)$ for $i=1,2$ and $j=2,3,4$. By Corollary~\ref{cor:critical_mult} we obtain the first two inequalities in \eqref{ineqs:high-mults2}. Moreover, by Lemma~\ref{lem:decreasing-ells} we have  
    \[n_4=\sum\limits_{\substack{i\in 3\bZ+1,\\i\ge 4}}\ell_i(T)<2\sum\limits_{j = 4}^{H + 1} \ell_j(T)\leq\ell_3(T).\qedhere\]
\end{proof}

\begin{remark}\label{rk:rn}
Let $\Lambda\in \mathcal{B}$. For $n> 2$, we have seen that $\alpha_n$ and $\beta_n$ are both roots of the characteristic polynomial $p_n(x)$ of $C_n^\Lambda$. The remaining eigenvalues of $C_n^\Lambda$ are therefore the roots of the degree $n-2$ polynomial $r_n(x):=\frac{p_n(x)}{(x-\alpha_n)(x-\beta_n)}$. Since $r_n(x)$ is determined entirely by our choice of~$\Lambda$, by Theorem~\ref{thm:arbitrary M}, for any $H\ge1$, the entire spectrum of any matrix $A\in \PHset(C_{H+1}^\Lambda, T)$ is determined solely by $T$ and $\Lambda$.
\end{remark}

Corollary~\ref{cor:critical_mult} guarantees the existence of a matrix in $A \in \mathcal R(T)$ with high multiplicities for five eigenvalues. The path-to-hedge construction used to obtain $A$ implies further spectral and structural properties for $A$, as noted in the previous remark. In the following sections we will show that this construction is the only way to obtain the five critical multiplicities.  
  
\section{The invertible subtrees lemma and combinatorics}
\label{sec:comb}
When given a tree that is a hedge, the previous sections describe a strategy for constructing a matrix with high eigenvalue multiplicities.  In this section, we approach the problem from the other direction, namely, we use the structure of the tree and a given unordered multiplicity list with some high multiplicities to infer necessary structural properties of any matrix achieving the list.  The major results in this section are Lemma \ref{lem:nullity-bound} and its corollaries, which show that matrices that achieve the maximum multiplicity for a tree cannot have certain types of invertible principal submatrices. The proscribed induced subtrees are based on minimal path covers of the original tree, so the goal of the rest of the section is to prove Lemmas \ref{lem:M} and \ref{lem:Sigma}, which determine the size and composition of minimal path covers for lush hedges. %

These results follow from two standard tools in the literature, namely zero forcing and the connection between path covers and maximum multiplicity in trees. We start by reviewing some relevant results.

 Zero forcing was introduced in 2008 in \cite{MR2388646} as a way of providing a combinatorial upper bound on the maximum multiplicity of a graph and has been used productively in many contexts; see, for example \cite{z-paper1, loop-zero-forcing}. For the convenience of the reader, we give a brief account here of the zero forcing concepts we will use. Given a graph $G=(V,E)$, we color each vertex either blue or white and write $B\subseteq V$ for the set of blue vertices.  The color change rule says that a blue vertex $v$ with exactly one white neighbor $u$ can \emph{color} or \emph{force} $u$ to change from white to blue and is written $v\to u$.  
A finite sequence of successive forces is called a \emph{zero forcing process}. %
The \emph{derived set} of $B$ is the set of vertices $B'\subseteq V$ which are blue after all possible applications of the color change rule are performed, and it does not depend on the particular zero forcing process chosen.  When $B'=V$, we say $B$ is a \emph{zero forcing set} of $G$. The minimum size of a zero forcing set is called the \emph{zero forcing number} of $G$ and is written $Z(G)$.

The following two results appear in \cite{MR2388646}, which uses them to prove that $M(G)\leq Z(G)$.  The main structural result of this section, Lemma \ref{lem:nullity-bound}, uses these same results, and so they are stated here for later reference.  %

\begin{lemma}[\cite{MR2388646}, Proposition 2.3]\label{lem:ZFvec}
Let $B$ be a zero forcing set of $G=(V,E)$ and $A\in\S(G)$.  If ${\bf v}\in \Ker(A)$ and ${\bf v}[B]= {\bf 0}$, then ${\bf v}={\bf 0}$.
\end{lemma}

\begin{lemma}[\cite{MR2388646}, Proposition 2.2]\label{lem:Zbound}
Let $G=(V,E)$, $|V|=n$, $A$ be an $n\times n$ matrix, and $N\subseteq [n]$ be a set of indices.  If the only vector ${\bf v}\in \Ker(A)$ with ${\bf v}[N] = {\bf 0}$ is the zero vector, then $\nullity(A)\leq |N|$.
\end{lemma}

Next we turn to the connection between zero forcing and path covers.  First, notice that each vertex in~$G$ can force at most one other vertex in a zero forcing process. Given a zero forcing process starting with a zero forcing set $B$ and $v_1\in B$, suppose $v_1\to v_2$, $v_2\to v_3$, \ldots, $v_{k-1}\to v_k$ where the process ends without $v_k$ forcing.  Then the set $\{v_1,v_2,\ldots,v_k\}$ is the vertex set of an induced path in~$G$ and is called a \emph{zero forcing chain}.  Notice that when a zero forcing process forces all of $V$, the set of zero forcing chains forms a path cover.  This shows that in general $P(G)\leq Z(G)$, but when $T$ is a tree, the authors of \cite{MR2388646} showed that every path cover is the set of zero forcing chains for some zero forcing process, and so $P(T)=Z(T)$.  Thus, the path cover number, the maximum nullity, and the zero forcing number coincide for trees.

\subsection{The invertible subtrees lemma} %
With the preliminaries done, we turn to results that will allow us to prove that in some cases part of the structure of the realizing matrix is implied from the prescribed multiplicities. As this paper is only concerned with trees, all the relevant results are stated for trees only.  We note that those results can be generalized to graphs which are not necessarily trees at the cost of making the statements more complicated. We point out where generalizations are possible.

\begin{definition}
Let $T$ be a forest, with subtrees $T_1$ and $T_2$. We say $T_1$ and $T_2$ are \emph{independent in~$T$} if $u \in V(T_1)$ and $v \in V(T_2)$ implies $\{u,v\} \not\in E(T)$. 
\end{definition}

The following result is central to the proof of the main theorem of this paper. We call it the Invertible Subtrees Lemma, and we believe it may also be useful in other contexts.

\begin{lemma}[Invertible Subtrees Lemma]\label{lem:nullity-bound}
  Let $T$ be a forest, $A\in\cs(T)$ and $c\ge1$.
  If $T_1,\dots,T_c$ are mutually independent subtrees of~$T$ and $A[T_i]$ is invertible for each $i\in [c]$, then the nullity of~$A$ is at most the path cover number of $T\setminus (T_1\cup \dots\cup T_c)$.
\end{lemma}
\begin{proof}
Let $H_1,\dots,H_k$ be the connected components of $H=T\setminus(T_1\cup\dots\cup T_c)$ and let $N(T_i)$ denote the set of neighbors of $T_i$ in~$T$, namely the set of all vertices in $V(T)\setminus V(T_i)$ which are adjacent to some vertex of $T_i$ in~$T$.  As the $T_i$ are mutually independent, $N(T_i)\subseteq V(H)$ for all $i\in [c]$.

Let $B$ be a zero forcing set for $H$ and let $B'$ be the derived set under a zero forcing process performed in~$T$. 
We claim that there is some $i_0\in [c]$ so that $N(T_{i_0})\subseteq B'$.

To see this, let $G$ be the graph obtained by contracting each $T_i$ and each $H_j$ to a single vertex. That is, $G$ is the graph with the $c+k$ vertices $\{T_1,\dots,T_c,H_1,\dots,H_k\}$, and $G$ has edges \[\{\{X,Y\}: \text{$T$ contains a (necessarily unique) edge between $X$ and $Y$}\}.\]
Then $G$ is a forest, and since $T_1,\dots,T_c$ are independent in~$T$, the partition of the vertices of $G$ into $\{T_1,\dots,T_c\}$ and $\{H_1,\dots,H_k\}$ is a bipartition of $G$.%

 Label each edge $\{T_i,H_j\}$ of $G$ with the integer $|B'\cap N(T_i)\cap V(H_j)|\in \{0,1\}$. Hence, the edge $\{T_i,H_j\}$ in~$G$ is labeled $1$ precisely when the vertex in $V(H_j)$ incident to $T_i$ lies in $B'$. Our claim is equivalent to the statement that there is some $i_0\in [c]$ so that every edge of $G$ incident to $T_{i_0}$ is labelled~$1$.

If the claim is false, then for every $i\in [c]$, there is some edge of $G$ incident to $T_i$ which is labelled~$0$.  On the other hand, the only forces $u\to v$ which are valid in $H$ but not $T$ are those where $u$ has a neighbor in one of the $T_i$. Since $B$ is a zero forcing set for $H$, it follows that every non-isolated $H_j$ in~$G$ has at least one incident edge labelled~$1$. 
Hence, we can find a path in~$G$ of the form $T_{i_1}$---$H_{j_1}$---$T_{i_2}$---$H_{j_2}$---\dots{} in which the edge labels alternate $0,1,0,1,\dots$. Because of this alternating labelling, we have $i_{t+1}\ne i_t$ and $j_{t+1}\ne j_t$ for each $t$. Since $G$ is acyclic, this implies that $i_s\ne i_t$ and $j_s\ne j_t$ for $s\ne t$, so we have an infinite path in the forest $G$, which is a contradiction. This establishes the claim.

Suppose without loss of generality that $i_0=c$; then $N(T_{c})\subseteq B'$. If ${\bf v}$ is a null vector for $A$ with ${\bf v}[B]={\bf 0}$, then the proof of Lemma~\ref{lem:ZFvec} implies that ${\bf v}[B']={\bf 0}$. So, working with the block decomposition of matrices and vectors given by the partition $V(T)=V(T_c)\cup (B'\setminus V(T_c))\cup (V(T)\setminus (B'\cup V(T_c)))$, we have
\[ {\bf 0}=A{\bf v}=\begin{pmatrix}
  A[T_c] & {*} & 0\\ {*} & {*} & {*} \\ 0&{*}&{*}
\end{pmatrix}\begin{pmatrix}
  {\bf v}[T_c]\\{\bf 0}\\{*}
\end{pmatrix}=\begin{pmatrix}
  A[T_c]{\bf v}[T_c]\\{*}\\{*}
\end{pmatrix}\]
(where $*$ indicates an entry which is not relevant to our argument), so ${\bf v}[T_c]={\bf 0}$ by the invertibility of $A[T_c]$.
If $c>1$, then working in the forest $T'=T\setminus T_c$, we can repeat the above argument to see that (for the same set $B$) there is $i_0'\in [c-1]$, say $i_0'=c-1$, so that ${\bf v}[B\cup V(T_{c-1}\cup T_c)]={\bf 0}$. Continuing inductively, we obtain ${\bf v}[B\cup V(T_1\cup \dots\cup T_c)]={\bf 0}$. Since $B$ is a zero forcing set for $H=T\setminus(T_1\cup\dots\cup T_c)$, Lemma~\ref{lem:ZFvec} gives that ${\bf v}={\bf 0}$ and so $\nullity(A)\leq |B|$ by Lemma \ref{lem:Zbound}. Since path covers and zero forcing sets coincide for forests \cite{MR2388646}, $|B|\leq P(T\setminus (T_1\cup \dots\cup T_c))$ and the result follows.
\end{proof}

Implications for matrices achieving the maximum eigenvalue multiplicity are stated below.

\begin{corollary}\label{cor:invertibleSubtree}
Let $T$ be a tree with $A\in\cs(T)$ and let $T_1,\ldots,T_c$ be mutually independent subtrees of~$T$.
\begin{enumerate}
    \item If $\nullity(A)>P\left(T\setminus (T_1\cup \dots\cup T_c)\right)$, then at least one of the principal submatrices $A[T_1],\ldots,A[T_c]$ is not invertible.
    \item If $\lambda\in\spec(A)$ with $\mult(\lambda)>P\left(T\setminus (T_1\cup \dots\cup T_c)\right)$, then at least one of the principal submatrices $A[T_1],\ldots,A[T_c]$ has $\lambda$ as an eigenvalue.
\end{enumerate}
\end{corollary}
\begin{proof}
The first part of the corollary follows directly from Lemma \ref{lem:nullity-bound}. %
The second part follows from the first, by considering $A':=A-\lambda I$.
\end{proof}

The invertible subtrees lemma is most useful when the path cover number of the graph decreases when the $T_i$ are removed.
\begin{observation}\label{obs:QinG}
Let $G$ be a graph and let $Q$ be a path in a path cover of size $P(G)$. Then $P(G\setminus Q)=P(G)-1$.
\end{observation}
\begin{proof}
  To see this, first note that $P(G\setminus Q)\leq P(G)-1$ because the induced path cover without $Q$ on $G\setminus Q$ has $P(G)-1$ paths.  Then, if $P(G\setminus Q)=k<P(G)-1$, a path cover of $G$ with $k+1<P(G)$ paths can be constructed by taking the union with $Q$.
\end{proof}

The special case where $c=1$ and $T_1$ is a single vertex in a minimal path cover will be used repeatedly and so it is stated as a corollary.

\begin{corollary}\label{cor:diagonalEntry}
Let $T$ be a tree, and suppose that $A=(a_{i,j})\in\cs(T)$ with $\lambda \in \spec(A)$ and $\mult(\lambda,A)=M(T)$.  
Then every vertex $v\in V(T)$ that is a singleton in some minimal path cover of~$T$ has 
 $a_{v,v}=\lambda$.
\end{corollary}
\begin{proof}
As a minimal path cover of~$T$ has a singleton path $\{v\}$, we have $M(T)=P(T)=P(T\setminus \{v\})+1$ by Observation \ref{obs:QinG}.  The result follows from setting $T_1=\{v\}$ in the second part of Corollary \ref{cor:invertibleSubtree}.
\end{proof}

\begin{remark}
We note here that versions of the previous results still hold on classes of graphs which are more general than trees, but the formal statements are considerably more complicated and are not needed for our main results, so we omit them.  In particular, the proof of Lemma \ref{lem:nullity-bound} only requires that the contracted graph $G$ is a tree and does not have multiple edges.  Similarly, the hypotheses of Corollary \ref{cor:diagonalEntry} can be modified to use loop zero forcing (see~\cite{loop-zero-forcing}) at the vertex $v$ as long as $M(G)=Z(G)$ and $v$ is not part of any cycle in the original graph. 
\end{remark}

\subsection{Path cover formulae} %
 Let $T$ be a lush hedge of height~$H$, and recall the chain of subtrees \[T=T^{(0)}\supseteq T^{(1)}\supseteq\dots\supseteq T^{(H)}=P_{H+1}\] from Definition~\ref{def:chain-subtrees}.  In order to determine the path cover numbers of these subtrees and identify paths in minimal path covers of $T^{(h)}$ which can be used in Corollary~\ref{cor:invertibleSubtree} and Corollary~\ref{cor:diagonalEntry}, we will use a result from Nylen in \cite{Nylen}.
 
 \begin{definition}
   Let $T=(V,E)$ be a tree.  A vertex $v\in V$ is called \emph{appropriate} if its
deletion from G has at least 2 components that are paths joined at the end to the deleted vertex. In other words, $v$ has
degree at least $3$ and has at least two adjacent pendent paths.
 \end{definition}
 
 \begin{lemma}[\cite{Nylen}]\label{lem:AppVertex}
   If $v$ is an appropriate vertex in a tree $T$ with pendent paths $p_1\text{---}\ldots\text{---}p_k$ and $q_1\text{---}\ldots\text{---}q_l$ where $v$ is adjacent to $p_1$ and $q_1$, then there is a minimal path cover of~$T$ which contains the path \[p_k\text{---}\ldots\text{---}p_1\text{---}v\text{---}q_1\text{---}\ldots\text{---}q_l.\]
 \end{lemma}
 
 We can now establish a formula for the path cover number of $T^{(h)}$. Note that the same formula appeared in Corollary~\ref{cor:critical_mult}; this will be important in the next section.

\begin{lemma}\label{lem:M} If $T$ is a lush hedge of height $H\ge 1$, and $0\le h\le H$, then
  \[M(T^{(h)})=P(T^{(h)})=\sum_{\substack{i\in 2\bZ+h+1,\\i\ge h+1}}\ell_i(T).\]
  Moreover, if $Q$ is either
  \begin{enumerate}
      \item a \pendentpath{(h+1)} in $T^{(h)}$, or
      \item %
      a singleton vertex $v \in V(T^{(h)})$ with $\h(v) \geq h$ and $\h(v)$ of the same parity as $h$,
  \end{enumerate}
  then there is a minimal path cover for $T^{(h)}$
  which contains $Q$ as one of the paths,
  and so \[P(T^{(h)}\setminus Q)=P(T^{(h)})-1.\]
\end{lemma}
\begin{proof}
First, we observe that for any tree $\widehat T$ with a pendent path $Q_{u,v}$ joining a non-leaf $u$ to a leaf $v$, the graph $\widehat T':=\widehat T\setminus(Q_{u,v}\setminus \{u\})$ obtained by contracting $Q_{u,v}$ to the vertex $u$ 
necessarily respects minimal path covers, and so $P(\widehat T)=P(\widehat T')$. Iterating this process starting with $\widehat T=T^{(h)}$, we obtain a lush hedge. Hence, to establish the path cover formula above, it suffices to prove it in the case $h=0$, and to prove statements (1) and~(2), it suffices to prove statement~(2) only in the case $h=0$ and for $Q=\{z\}$, where $z$ is a vertex with even height in $T=T^{(0)}$. 
We write %
$V_i:=V_i(T)$.

To this end, will show by induction on $H$ that there is a minimal path cover~$\mathcal{C}$ of~$T$ consisting only of $P_3$s and $P_1$s, where every vertex at an odd height together with two of its children forms a $P_3$ in the cover, and, for even $i\in [H]$, the cover $\mathcal{C}$ contains no edge between $V_{i-1}$ and $V_i$. %

If $H=0$, then $T=P_1$ and this is trivial.
If $H=1$, then $T=K_{1,n}$ is a star and the only vertex at an odd height is the root. A minimal path cover involves one $P_3$ from a leaf to the root and back to a leaf along with $P_1$s covering any remaining leaves, and $P(K_{1,n})=n-1=|V_0|-|V_1|=\ell_1(K_{1,n})$.

Next, suppose that $H\geq 2$. We will construct a minimal path cover ${\mathcal C}$ of~$T$, of the type described above.  To start, each vertex in $V_1$ is appropriate as it has at least two children which are pendent leaves. By Lemma \ref{lem:AppVertex} we can start constructing ${\mathcal C}$ %
by taking paths which start at a leaf, go up to their parent vertex at height~$1$, then go back down to a leaf not involving any vertices of height~$2$.  Doing this repeatedly, we cover every vertex at height~$1$. We also include any remaining leaves as singleton paths in the cover~${\mathcal C}$.  The vertices which remain uncovered are those at height~$2$ or greater and the process repeats inductively: the uncovered portion $T':=T\setminus (V_0\cup V_1)$ of the tree has a minimal path cover of the type claimed. Including this in ${\mathcal C}$ completes the construction. 

To establish the formula for $P(T)$, suppose $T$ is a lush hedge of height $H\ge 2$ and the formula for the path cover number has been established for lush hedges of height $H-2$. Then, in particular, $T'=T\setminus(V_0\cup V_1)$ has 
\[ P(T')=\sum_{1\le i\in 2\bZ+1}\ell_i(T')=\sum_{3\le i\in 2\bZ+1}\ell_i(T).\]
Moreover, since $T\setminus T' = T[V_0\cup V_1]$ is a union of stars, it is easy to see that
\[ P(T\setminus T') = \ell_1(T).\]
The minimal path cover $\mathcal{C}$ of~$T$ contains no edge between $T\setminus T'$ and $T'$, so
\[P(T)=P(T\setminus T') + P(T')=\ell_1(T)+\sum_{3\le i\in 2\bZ+1}\ell_i(T)= \sum_{1\le i\in 2\bZ+1}\ell_i(T),\]
  as required.

Now we turn to statement~(2) in the case $h=0$ and $Q=\{z\}$, where $z$ is either a leaf or a vertex at an even height. Consider the minimal path cover~$\mathcal{C}$ described above.  The result follows immediately if $z$ is already a singleton in~$\mathcal{C}$, so assume that it is not. The non-singleton paths in~$\mathcal{C}$ are all of the form $u\text{---}p\text{---}v$, where $u,v$ are siblings of even height and $p$ is their parent, so the path in $\mathcal{C}$ containing~$z$ is of this form, say $z\text{---}p\text{---}v$.
If $p$ has three or more children, then let $u'$ be a child other than $v$ or $z$. Then $u'$ is necessarily a singleton in $\mathcal{C}$, so we can replace the paths $u'$ and $z\text{---} p\text{---} v$ in $\mathcal{C}$ with $z$ and $u'\text{---} p \text{---} v$, to form a new path cover~$\mathcal{C}'$ of~$T$ in which $z$ is a singleton. Otherwise, $p$ is not the root of~$T$ (since the root has at least $3$ children), so let $g$ be the parent of $p$.  Since $g$ is at a positive even height, either $g$ is either a singleton path in~$\mathcal{C}$, or $g$ is the end of a path in~$\mathcal{C}$. In either case, we can form a new path cover $\mathcal{C}'$ of~$T$ by replacing the paths $\cdots \text{---} g$ and $z\text{---} p\text{---} v$ in~$\mathcal{C}$ with $\cdots \text{---} g\text{---} p\text{---} v$ and $z$. 

This shows that we can always find a path cover $\mathcal{C}'$ in which $z$ is a singleton, and so that $|\mathcal{C}'|=|\mathcal{C}|$,  so $\mathcal{C}'$ is also a minimal path cover of~$T$. By Observation~\ref{obs:QinG}, this implies that $P(T\setminus Q)=P(T)-1$.
\end{proof}

Notice that the eigenvalue $\beta_i$ recurs every third step in the $C_n^\Lambda$ sequence given in Proposition \ref{prop:interlacing}.  This is useful because a similar fact holds for the path covers in lush hedges.  This connection is made in the following definition and lemma, which give combinatorial interpretations of the second summation formula in Corollary~\ref{cor:critical_mult}.

\begin{definition}\label{def:M-hat}
  Let $T$ be a hedge of height $H$. For $0\le h\le H$, we write \[\widehat{M}(T^{(h)}):=\sum_{\substack{i\in 3\bZ+h+1,\\i\ge h+1}}\ell_i(T).\]
\end{definition}

\begin{remark}\label{rk:critical-M}
Using this notation and Lemma~\ref{lem:M}, the inequalities~\eqref{ineqs:high-mults2} defining the notion of a \highMultiplicityList{} may be written concisely as
\begin{align*}
\mult(\alpha_j,A)&\ge M(T^{(j-1)}),\;j=1,2,
\\\mult(\beta_j,A)&\ge\widehat M(T^{(j-1)}),\;j=2,3,4,\text{ and}
\\
\mult(\beta_4,A)&<\ell_3(T).
\end{align*}
\end{remark}

\begin{lemma}\label{lem:Sigma}
  Let $T$ be a lush hedge of height $H\ge 2$ and $0\le h\le H$. Let $Q$ be a \pendentpath{(h+1)} in $T^{(h)}$, and \[\widehat{V}:=\{v\in V(T^{(h)}): h+2\le \h(v)\in 3\bZ+h+2\}.\]
  Then
  \[ P(T^{(h)}\setminus (Q\cup \widehat{V}))=\widehat{M}(T^{(h)})-1\text{\quad and\quad}P(T^{(h)}\setminus \widehat{V})=\widehat{M}(T^{(h)}).\]
\end{lemma}
\begin{proof}
  Let us write $V_j:=V_j(T^{(h)})$, so that $V_j\subseteq V_j(T)$ and, for $j\ge h$, we have $V_j=V_j(T)$. Then we have the decompositions
  \[ T^{(h)}\setminus \widehat{V}=T^{(h)}[V_0\cup \dots\cup V_{h+1}]\cup \bigcup_{h+4\le i\in 3\bZ+h+1} T^{(h)}[V_{i-1}\cup V_i]\]
  and
  \[ T^{(h)}\setminus(Q\cup \widehat{V})=(T^{(h)}[V_0\cup \dots\cup V_{h+1}]\setminus Q)\cup \bigcup_{h+4\le i\in 3\bZ+h+1} T^{(h)}[V_{i-1}\cup V_i].\]
  The path cover number is additive over disjoint unions, so it suffices to compute the path cover number of each term in these decompositions.
  
  We claim that for $i\ge h+1$,
  \[ P(T^{(h)}[V_{i-1}\cup V_i])= \ell_i.\]
  To see this, note that if $h+1\le i\le H$, then $V_{i-1}$ and $V_i$ are non-empty, and $T^{(h)}[V_{i-1}\cup V_i]$ is of the form $\bigcup_{1\le j\le |V_i|}K_{1,d_j}$, where $d_j\ge 3$ (since $T$ is lush), which has path cover number $\sum_{j=1}^{|V_j|}(d_j-1)=\ell_j$. To verify the remaining degenerate cases: if
  $i>H+1$, then $V_{i-1}=V_i=\emptyset$ and $\ell_i=0$; and if $i=H+1$, then $V_{i-1}=\{v\}$ contains only the root vertex of~$T$ and $V_i=\emptyset$ and $\ell_i=1$.

  As observed 
  in the proof of Lemma~\ref{lem:M} above, adding a pendent path to a tree at a leaf does not change its path cover number. Hence, $P(T^{(h)}[V_0\cup \dots \cup V_{h+1}])=P(T^{(h)}[V_{h}\cup V_{h+1}]=\ell_{h+1}$. It now follows that $P(T\setminus \widehat{V})=\widehat{M}(T^{(h)})$.

  Finally, $P(T^{(h)}[V_0\cup \dots\cup V_{h+1}]\setminus Q)=P(T^{(h)}[V_{h}\cup V_{h+1}]\setminus \{u\})$ where $u$ is the vertex of $Q$ in $V_{h}$. Now $T^{(h)}[V_{h}\cup V_{h+1}]\setminus \{u\}=\bigcup_{j=1}^{|V_{h+1}|}K_{1,d_j'}$ where $d_j'\ge 2$ (since $T$ is lush) and $\sum_{j=1}^{|V_{h+1}|} d_j'=|V_h|-1$. Since $d_j'\ge 2$ for all $j$, this graph has path cover number $\sum_{j=1}^{|V_{h+1}|} (d_j'-1) = \ell_{h+1}-1$. It follows that $P(T^{(h)}\setminus (Q\cup \widehat{V}))=\widehat{M}(T^{(h)})-1$.
\end{proof}

\goodbreak

\goodbreak

\section{The main theorem}
\label{sec:main-theorem}

Recall that for a lush ledge $T$, Corollary~\ref{cor:critical_mult} provides a matrix $A \in \mathcal R(T)$ with a \highMultiplicityList{} (see Definition~\ref{def:critical-multiplicity2}), which, loosely speaking, entails that five eigenvalues of~$A$ have high multiplicities. The following theorem is the main result of the paper, in which we show that every matrix achieving such a \highMultiplicityList{} $\ordm$ arises from path-to-hedge construction and thus its entire spectrum is determined by those five selected eigenvalues and their high multiplicities. 

\begin{theorem}\label{thm:main}
 Let $T$ be a lush hedge of height $H\geq 2$ and $A\in \cs(T)$ with multiplicity list ${\ordm}$ that is a \highMultiplicityList{} for $T$.  Then $A\in \PHset(C_{H+1}^{\Lambda},T)$ for some $\Lambda\in {\mathcal B}$. 
Furthermore, the entire spectrum of every such matrix $A$ is determined solely by~$\Lambda$:
  \[\spec(A)=\bigcup_{i=1}^{H+1} \ell_i(T)\spec(C_i^{\Lambda}).\]
  In particular, we have equality in the first two inequalities in~\eqref{ineqs:high-mults2} for~${\ordm}$.
\end{theorem}
Since $\S(T)\subseteq \R(T)$, this theorem also applies in the more standard context of symmetric matrices with graph $T$.

In the rest of this section we prove Theorem~\ref{thm:main}. We will first establish spectral conditions which imply that a matrix $A \in \cs(T)$ necessarily comes from the path-to-hedge construction.
Our first lemma gives information on diagonal entries and \pendentpath{2}s. 

\begin{lemma}\label{prop:alpha-vertices}
  Let $T$ be a lush hedge, $A=(a_{i,j})\in\cs(T)$ and $\alpha_1,\alpha_2$ distinct real numbers with
  \begin{equation*}
    \mult(\alpha_j,A)\ge M(T^{(j-1)})
  \end{equation*}
    for $j\in[2]$.
  \begin{enumerate}[(C1)]
  \item\label{claim1} If $v\in V(T)$ with $\h(v)\ne 1$, then
    \[ a_{v,v}=
      \begin{cases}
        \alpha_1&\text{if $\h(v)$ is even,}\\
        \alpha_2&\text{if $\h(v)$ is odd.}
      \end{cases}
    \]
\item\label{claim3} We can collapse leaves in $A$ (as in Definition~\ref{def:collapsing-trees-matrices}) to obtain a matrix $A^{(1)}\in \cs(T^{(1)})$, and $\alpha_2\in \spec(A^{(1)}[Q])$ for every \pendentpath{2} $Q$ in $T^{(1)}$.
  \end{enumerate}
\end{lemma}
\begin{proof}
  By Corollary~\ref{cor:diagonalEntry} and~Lemma~\ref{lem:M}, we have $a_{v,v}=\alpha_1$ whenever $v$ is a vertex of even height. This includes all leaf vertices, so we can collapse leaves in $A$ to obtain $A^{(1)}\in \cs(T^{(1)})$. This collapsing removes $\ell_1$ leaves $u$ from $T$, where in each case $A[\{u\}]=(\alpha_1)$, so $\spec(A^{(1)})=\spec(A)\setminus \ell_1 \{\alpha_1\}$, by equation~\eqref{eq:collapsed-spectrum}. In particular, \[\mult(\alpha_2,A^{(1)})=\mult(\alpha_2,A)\ge M(T^{(1)}).\] %
  Applying Corollary~\ref{cor:diagonalEntry} and Lemma~\ref{lem:M} to $A^{(1)}$, we have $a^{(1)}_{v,v}=\alpha_2$ for every vertex $v$ of $T^{(1)}$ with $\h(v)$ odd and $\h(v)\ne 1$. Since $A$ is obtained by successive summand duplication of $A^{(1)}$ (which copies diagonal entries from diagonal entries of $A^{(1)}$ to diagonal entries of~$A$ at the same height), \ref{claim1} follows.

  If $Q$ is a \pendentpath{2} in $T^{(1)}$,
  then by Lemma~\ref{lem:M},
  \[P(T^{(1)}\setminus Q)<P(T^{(1)})=M(T^{(1)})\le \mult(\alpha_2,A^{(1)}).  \]
  By Corollary~\ref{cor:invertibleSubtree}, $\alpha_2$ is an eigenvalue
  of $A^{(1)}[Q]$, so \ref{claim3} holds.
\end{proof}

The next lemma underpins the inductive step of our main theorem, as it will allow us to perform repeated collapses at successively greater heights. 

\begin{lemma}\label{lem:inductive}
  Let $\alpha,\alpha',\beta$ be distinct real
  numbers, $1\le h\le H-1$, $T$ a
  lush hedge of height $H\geq 2$, and $w_h\in W(P_h)$. Let $B\in \cs(T^{(h)})$ with  \[\mult(\beta,B)\ge \widehat{M}(T^{(h)})\] and $b_{v,v}\ne \beta$ for all vertices $v$ of $T^{(h)}$ with $\h(v)\ge3$. Furthermore, suppose:
  \begin{enumerate}[(P1)]
  \item\label{1hypothesis1} Every \pendentpath{h} $Q$ in $T^{(h)}$ satisfies:
  \begin{itemize}
  \item $w_{B[Q]}=w_h$, 
  \item $\alpha \in \spec(B[Q])$, 
  \item and if $h>1$ then $\alpha'\in \spec(B[Q_0])$, where $Q_0$ is the \pendentpath{(h-1)} in $T^{(h)}$ obtained from $Q$ by removing a vertex ($v \in V(Q)$ of $\h(v)=h-1$).
  \end{itemize}
  \item\label{1hypothesis3} If $h=1$, then $\alpha'\in \spec(B[Q'])$ for every \pendentpath{2} $Q'$ in $T^{(1)}$; and if $h>1$, then $b_{u,u}=\alpha'$ for every vertex $u$ of $T^{(h)}$ of height $h$.
   \end{enumerate}
   Then there is a weight $w_{h+1}\in W(P_{h+1})$ so that for all \pendentpath{(h+1)}s $Q'$ in $T^{(h)}$, we have $w_{B[Q']}=w_{h+1}$, and $\alpha'$ and $\beta$ are both eigenvalues of $B[Q']$.
\end{lemma}

\begin{proof}
 Let $Q'$ be a \pendentpath{(h+1)} in $T^{(h)}$, joining a vertex $u$ to a \pendentpath{h} $Q$. Then: $$B[Q']=
  \left(\begin{smallmatrix}
      b_{u,u}& c_u{\bf e}_1\trans\\
      d_u{\bf e}_1& B[Q]
    \end{smallmatrix}\right)$$
   where ${\bf e}_1=(1,0,0,\dots,0)\trans\in \bR^{h}$ and $w_{B[Q]}=w_h$ by~\ref{1hypothesis1}.
   For $h-1\le i\le h+1$, let $\Delta_i$ be the characteristic polynomial of the trailing principal $i\times i$ submatrix of $B[Q']$, i.e $\Delta_{h+1}$ is the characteristic polynomial of $B[Q']$, $\Delta_h$ is the characteristic polynomial of $B[Q]$, and $\Delta_{h-1}$ is the characteristic polynomial of $B[Q_0]$ as defined in the statement. (In the degenerate case, $\Delta_0:=1$.)
  
  We first show that $\alpha'\in \spec(B[Q'])$. If $h=1$, see~\ref{1hypothesis3}. If $h>1$, then by~\ref{1hypothesis3} again, $b_{u,u}=\alpha'$,
  so
  \begin{equation}\label{eq:1Delta}
    \Delta_{h+1}(x)=(x-\alpha')\Delta_{h}(x)-c_ud_u\Delta_{h-1}(x)
  \end{equation}
  and $\Delta_{h-1}(\alpha')=0$ by~\ref{1hypothesis1}, so $\Delta_{h+1}(\alpha')=0$ by \eqref{eq:1Delta}, as required. 

  We will now apply Corollary~\ref{cor:invertibleSubtree} to show that $\beta\in \spec(B[Q'])$. Consider \[\widehat{V}=\{v\in V(T^{(h)}): h+2\le \h(v)\in 3\bZ+h+2\}.\]
  Note that $\{Q'\}\cup \{\{v\}:v\in \widehat{V}\}$ is an independent set of paths in $T^{(h)}$; see Figure~\ref{fig:induction}, where these paths are colored orange.  
  By Lemma~\ref{lem:Sigma}, we have \[P(T^{(h)}\setminus (Q'\cup \widehat{V}))=\widehat{M}(T^{(h)})-1<\widehat{M}(T^{(h)})\le\mult(\beta,B).\]  %
Now Corollary~\ref{cor:invertibleSubtree} together with the assumption $b_{v,v} \neq \beta$ for $v \in \hat V$ immediately implies that $\beta\in \spec(B[Q'])$.

  \inductiveStepPicture 

Finally, we must show that $w_{B[Q']}$ is independent of our choice of $Q'$. By~\ref{1hypothesis1}, we already know that $w_{B[Q]}=w_h$ does not depend on $Q'$, and if $h>1$, then $b_{u,u}=\alpha_{h+1}$ by~\ref{1hypothesis3}. If $h=1$, then the weight $w_1$ from \ref{1hypothesis1} takes a single value, say $b_1$, so the trace of $B[Q']$ is equal to $b_1+b_{u,u}$. On the other hand, $B[Q']$ has the spectrum $\{\alpha_2,\beta\}$, so $b_{u,u}=\alpha_2+\beta-b_1$.

It remains to show that the product $c_ud_u$ is independent of $Q'$. Since $\Delta_{h+1}(\beta)=0$, $\alpha'\ne \beta$ by assumption, and $\Delta_{h}(\beta)\ne0$ by 
Lemma~\ref{lem:steps}, equation~\eqref{eq:1Delta} gives  $c_ud_u=(\beta-\alpha_{h+1})\frac{\Delta_{h}(\beta)}{\Delta_{h-1}(\beta)}$, completing the proof that $w_{B[Q']}$ is independent of $Q'$.
\end{proof}

\begin{proof}[Proof of Theorem~\ref{thm:main}]
  It is %
  convenient to define $M(T^{(i)})=\widehat{M}(T^{(i)})=0$ for $i>H$.
  
 Since $A$ has a \highMultiplicityList{}, there are five distinct real numbers $\alpha_1,\alpha_2,\beta_2,\beta_3,\beta_4$ with multiplicities $m_1,m_2,n_2,n_3,n_4\ge0$ in $\sigma(A)$
 satisfying~\eqref{ineqs:high-mults2}. 
 In view of Remark~\ref{rk:critical-M}, we have
 \begin{gather*}
     \mult(\alpha_j,A)\ge M(T^{(j-1)}),\quad j=1,2,\text{ and}\\
     \mult(\beta_j,A)\ge \widehat M(T^{(j-1)}),\quad j=2,3,4.
 \end{gather*}
The inequalities for $\mult(\alpha_j,A)$ allow us to use Lemma \ref{prop:alpha-vertices} for $A$. Hence, $a_{v,v}=\alpha_1$, if $\h(v)$ is even,  $a_{v,v}=\alpha_2$, if $\h(v)$ is odd and $\h(v)\neq 1$. Note that this fixes the corresponding diagonal entries of any matrix obtained by collapsing from $A$, since collapsing does not change the value of non-deleted diagonal entries. Lemma \ref{prop:alpha-vertices} allows the first collapse of~$A$ to obtain $A^{(1)} \in \R(T^{(1)})$ with all the diagonal entries except those corresponding to vertices at height $1$ equal to either $\alpha_1$ or $\alpha_2$ as above. 
We have
  $\spec(A^{(1)})=\spec(A)\setminus \ell_1\{\alpha_1\}$ by Lemma~\ref{prop:alpha-vertices} and equation~\eqref{eq:collapsed-spectrum} on page~\pageref{eq:collapsed-spectrum},
  so 
  \[\mult(\beta_j,A^{(1)})=\mult(\beta_j,A),\qquad  j=2,3,4.\]
  
Now our aim is to show that we can repeatedly collapse $A^{(h-1)} \in \R(T^{(h-1)})$ to  $A^{(h)} \in \R(T^{h})$ using Lemma~\ref{lem:inductive}. From above we already know that \ref{1hypothesis3} for $h>1$ in this lemma will be satisfied for any matrix $B \in \R(T^{(h)})$ obtained by repeated collapsing from $A$, where $\alpha'$ alternates between $\alpha_1$ for $h$ odd  and $\alpha_2$ for $h$ even. The technical condition $b_{v,v} \neq \beta$ if $\h(v)\geq 3$ also clearly holds.  For $h=1$,  \ref{1hypothesis3} is satisfied by \ref{claim3} in Lemma~\ref{prop:alpha-vertices}.

Let us write $\Lambda_1=(\alpha_1)$, $\Lambda_2=(\alpha_1,\alpha_2,\beta_2)$, $\Lambda_3=(\alpha_1,\alpha_2,\beta_2,\beta_3)$ and $\Lambda_j=\Lambda=(\alpha_1,\alpha_2,\beta_2,\beta_3,\beta_4)$ for $j\ge4$. We also write ${\mathcal B}_j=\mathcal{B}$ for $j\ge 4$. We claim that  for $h=1,\dots,H$ there exist matrices $A^{(h)}\in \cs(T^{(h)})$ so that:
  \def\hcounter{h}
  \newcommand{\Sref}[2][h]{{\def\hcounter{#1}\ref{S-#2}}}
  \begin{enumerate}[(S1)${}_{\protect\hcounter}$]
  \item\label{S-collapse1} $A^{(0)}=A$, and $A^{(h)}$ is
    obtained by collapsing \pendentpath{h}s in $A^{(h-1)}$.
  \item\label{S-mults1} 
    $\mult(\beta_j,A^{(h)})\ge \widehat{M}(T^{(j-1)})$ for
    $j=h+1,h+2,h+3$.
      \item\label{S-pendent1} $\Lambda_{h+1}\in {\mathcal B}_{h+1}$ and $w_{A^{(h)}[Q_{h+1}]}=w_{C_{h+1}^{\Lambda_{h+1}}}$ for every \pendentpath{(h+1)} $Q_{h+1}$ in $T^{(h)}$.
  \end{enumerate}

 We have already established above that \Sref[1]{collapse1} and \Sref[1]{mults1} hold. 
 Now we can apply Lemma~\ref{lem:inductive} with $h=1$, $\beta=\beta_2$, $B=A^{(1)}$, hence  $\beta_2\in \spec(A^{(1)}[Q'])$ for any pendent $2$-path $Q'$ in $T^{(1)}$. Since we already know that $\alpha_1$ is a diagonal element of $A^{(1)}[Q']$ and  $\alpha_2\in \spec(A^{(1)}[Q'])$, we have $w_{A^{(1)}[Q']}=w_{C_2^{\Lambda_2}}$ where $\Lambda_2=(\alpha_1,\alpha_2,\beta_2)\in \mathcal{B}_2$, by Proposition~\ref{prop:interlacing}, proving \Sref[1]{pendent1}. %
 
  Now we want to establish \Sref[2]{collapse1}--\Sref[2]{pendent1}. By \Sref[1]{pendent1}, we can collapse \pendentpath{2}s in $A^{(1)}$ to obtain $A^{(2)}\in \cs(T^{(2)})$, so \Sref[2]{collapse1} holds.
  By Proposition~\ref{prop:interlacing},
 $\spec(C_2^\Lambda)=\{\alpha_2,\beta_2\}$,
  hence by \Sref[1]{pendent1} and equation~(\ref{eq:collapsed-spectrum}):
  \begin{gather*}%
    \mult(\beta_2,A^{(2)})=\mult(\beta_2,A^{(1)})-\ell_2,\quad\text{and}\\
    \mult(\lambda,A^{(2)})=\mult(\lambda,A^{(1)}),\quad \lambda\in  \{\beta_3,\beta_4\}.
  \end{gather*}
  Observing $\beta_{5}=\beta_2$ and $\widehat{M}(T^{(4)})=\widehat{M}(T^{(1)})-\ell_2$,~\Sref[2]{mults1} now follows directly from~\Sref[1]{mults1}. 
  
  Next, observe that the hypotheses of Lemma~\ref{lem:inductive} are satisfied (for $h=2$) by $\beta:=\beta_{3}$, $\alpha:=\alpha_{1}=\alpha_3$  and $B:=A^{(2)}$. Indeed, we have $\mult(\beta_{3},A^{(2)})\ge \widehat{M}(T^{(2)})$ by \Sref[2]{mults1}; and \ref{1hypothesis1} follows from \Sref[1]{pendent1}.
  By Lemma~\ref{lem:inductive}, there is a matrix $B_{3}$ with  $\sigma(B_3)=\{\alpha_{1},\beta_{3},\alpha_2+\beta_2-\beta_3\}$ so that $w_{A^{(2)}[Q]}=w_{B_{3}}$ whenever $Q$ is a \pendentpath{3} in $T^{(2)}$. By~\Sref[1]{pendent1},  $w_{B_{3}(1)}=w_{C_2^{\Lambda_2}}$, so by Proposition~\ref{prop:interlacing} and its proof, we have $\Lambda_3\in\mathcal{B}_3$ and $w_{B_{3}}=w_{C_{3}^{\Lambda_3}}$, so \Sref[2]{pendent1} holds. Note also that $\sigma(C_3^{\Lambda_3})=\sigma(B_3)=\{\alpha_1,\beta_3,\alpha_2+\beta_2-\beta_3\}$.

  If $H=2$, we are done. Otherwise, $H\ge 3$, and before proceeding inductively, we wish to show that $\alpha_2+\beta_2\ne\beta_3+\beta_4$, in order to apply Proposition~\ref{prop:interlacing}. Suppose instead (for a contradiction) that $\alpha_2+\beta_2=\beta_3+\beta_4$; then $\beta_4=\alpha_2+\beta_2-\beta_3\in \sigma(C_3^{\Lambda_3})$, as noted above.  By \Sref[2]{pendent1}, we can collapse \pendentpath{3}s in $A^{(2)}$ to obtain $A^{(3)}\in \cs(T^{(3)})$.
  By equation~(\ref{eq:collapsed-spectrum}), we have 
  \[
  \ell_3\sigma(C_3^{\Lambda_3})\subseteq \sigma(A^{(2)}),
  \] so by \eqref{ineqs:high-mults2},
  \begin{gather*}
    \ell_3\le \mult(\beta_4,A^{(2)})= \mult(\beta_4,A)=n_4<\ell_3
  \end{gather*}
  a contradiction.
  Hence, we have $\alpha_2+\beta_2\ne \beta_3+\beta_4$.
  We remark, that up to this point we did not need the hypothesis on the multiplicity of $\beta_4$, and we have proved there exists an eigenvalue $\alpha_2+\beta_2-\beta_3$ in $A$ of multiplicity $m'$, and $m'\ge \ell_3$. %
  Moreover, if $\lambda\not\in \{\alpha_1,\alpha_2,\beta_2,\beta_3,\alpha_2+\beta_2-\beta_3\}$, then $\mult(\lambda,A)=\mult(\lambda,A^{(3)})\le M(T^{(3)})<\ell_3$ by equation~(\ref{eq:collapsed-spectrum}), Lemma~\ref{lem:decreasing-ells} and Lemma~\ref{lem:M}, so $\alpha_2+\beta_2-\beta_3$ is the unique eigenvalue of $A$ with multiplicity at least $\ell_3$, apart from $\alpha_1,\alpha_2,\beta_2,\beta_3$. This is summarized in Corollary~\ref{cor:four-multiplicities}.\label{proof:side-product} 

  Now suppose $3\le h\le H$, and $A^{(h-1)}\in \cs(T^{(h-1)})$ satisfies
  \Sref[h-1]{collapse1}, \Sref[h-1]{mults1} and \Sref[h-1]{pendent1}. 
  Since $\alpha_2+\beta_2\ne \beta_3+\beta_4$, by \Sref[h-1]{pendent1} and Proposition~\ref{prop:interlacing},
  \begin{equation*}\label{eq:spectrum-Lambda}\spec(C_h^{\Lambda_h})\cap \Lambda=\{\alpha_h,\beta_h\}.\end{equation*}
  Moreover, by \Sref[h-1]{pendent1}, we can collapse \pendentpath{h}s in $A^{(h-1)}$ to obtain $A^{(h)}\in \cs(T^{(h)})$, so \Sref[h]{collapse1} holds.
  By equation~(\ref{eq:collapsed-spectrum}) on page~\pageref{eq:collapsed-spectrum},
  \begin{gather*}
    \mult(\beta_h,A^{(h)})=\mult(\beta_h,A^{(h-1)})-\ell_h,\quad  \quad\text{and}\\
    \mult(\lambda,A^{(h)})=\mult(\lambda,A^{(h-1)}),\quad \lambda\in \{\beta_{h+1},\beta_{h+2}\}.
  \end{gather*}
   Now~\Sref[h]{mults1} follows from~\Sref[h-1]{mults1} by observing that 
$M(T^{(h+1)})=M(T^{(h-1)})-\ell_h$, $\beta_{h+3}=\beta_h$ and $\widehat{M}(T^{(h+2)})=\widehat{M}(T^{(h-1)})-\ell_h$. 
Next, observe that the hypotheses of Lemma~\ref{lem:inductive} are satisfied (for the same value of $h$) for $B:=A^{(h)}$, $\alpha:=\alpha_h$ and $\alpha':=\alpha_{h-1}=\alpha_{h+1}$, and $\beta:=\beta_{h+1}$. Indeed, we have $\mult(\beta_{h+1},A^{(h)})\ge \widehat{M}(T^{(h)})$ by \Sref[h]{mults1}, and \ref{1hypothesis1} follows from \Sref[h-1]{pendent1}.
  By Lemma~\ref{lem:inductive}, there is a matrix $B_{h+1}$ with eigenvalues $\alpha_{h+1},\beta_{h+1}$ so that $w_{A^{(h)}[Q]}=w_{B_{h+1}}$ whenever $Q$ is a \pendentpath{(h+1)} in $T^{(h)}$. By~\Sref[h-1]{pendent1},  $w_{B_{h+1}(1)}=w_{C_h^{\Lambda_h}}$, so by Proposition~\ref{prop:interlacing}, we have $\Lambda_{h+1}=\Lambda\in {\mathcal B}_{h+1}=\mathcal{B}$ and $w_{B_{h+1}}=w_{C_{h+1}^\Lambda}$, so \Sref[h]{pendent1} holds. 
  This establishes our claim.

Using \Sref[h]{collapse1} for $h\in [H]$, Proposition~\ref{prop:collapsing} shows that $A\in \PHset(A^{(H)},T)$. Note that if $H+1<4$, we can extend $\Lambda_{H+1}\in\mathcal{B}_{H+1}$ to some $\Lambda\in \mathcal{B}$, and if $H+1\ge 4$, then we already have $\Lambda\in \mathcal{B}$. By \Sref[H]{pendent1}, we have $w_{A^{(H)}}=w_{C^\Lambda_{H+1}}$, so $A\in \PHset(C^\Lambda_{H+1},T)$ as required.

  The expression for $\sigma(A)$ was established in Theorem~\ref{thm:arbitrary M}. It follows directly from Corollary~\ref{cor:critical_mult} that the first two inequalities in \eqref{ineqs:high-mults2} are equalities. 
  
\end{proof}

Recall that for the smallest lush hedge $\bftree$ of height $2$, the fifth eigenvalue $\lambda_5$ of the spectrum $\{\lambda_1,\lambda_2^{(2)},\lambda_3^{(4)},\lambda_4^{(2)},\lambda_5\}$ realized in $\S(\bftree)$ must fulfill the linear constraint $\lambda_5=\lambda_2+\lambda_4-\lambda_1$. The next corollary, a side-product of the first part of the proof of Theorem~\ref{thm:main} on page~\pageref{proof:side-product}, shows that four eigenvalues of high multiplicity of any lush hedge imply the same linear constraint on a fifth eigenvalue.

\begin{corollary}
 \label{cor:four-multiplicities}
   Let $T$ be a lush hedge of height $H\geq 2$ and $A\in \cs(T)$ with multiplicity list ${\ordm}$ that contains multiplicities $m_1$, $m_2$, $n_2$ and $n_3$ satisfying:
   \begin{equation*}\left\{\quad
  \begin{aligned}
    m_j &\ge \sum\limits_{\substack{i\in 2\bZ+j,\\ i\ge j}}
    \ell_i(T),\; j=1,2; \\
     n_j&\ge \sum\limits_{\substack{i\in 3\bZ+j,\\i\ge j}}
    \ell_i(T),\; j=2,3\\
   \end{aligned}\right.
  \end{equation*}
   Then $\ordm$ contains a unique multiplicity $m'$ (in addition to $m_1$, $m_2$, $n_2$ and $n_3$) satisfying $m' \geq \ell_3(T)$. Moreover, if the matrix $A$ has eigenvalues: $\alpha_1$ with multiplicity $m_1$, $\alpha_2$ with multiplicity $m_2$, $\beta_2$ with multiplicity $n_2$,  and $\beta_3$ with multiplicity $n_3$, then the unique eigenvalue of~$A$ (besides $\alpha_1$, $\alpha_2$, $\beta_2$, $\beta_3$) with multiplicity $m'$ is $\alpha_2+\beta_2-\beta_3$. 
\end{corollary}

We conclude this section with an application of Theorem \ref{thm:main} to a specific tree and multiplicity list, demonstrating a cubic constraint on the placement of the eigenvalues before showing that $\RS(T,{\ordm})$ need not be convex.

\begin{example}\label{ex:T31-M-construction}
Let $T$ be a lush hedge of height $H=3$; the smallest such tree is shown in Figure~\ref{fig:T31}. Write $\ell_i=\ell_i(T)$ for
$i \in [H + 1]$ and recall that $\ell_4=1$.

\begin{figure}[htb]
  \definecolor{green_}{rgb}{0,.9,.4} %
  \definecolor{orange_}{rgb}{.9,.2,0} %
  \definecolor{c0}{rgb}{0,0.1,.9} %
  \definecolor{c-12}{rgb}{0.9,0.1,0.8}  %
  \definecolor{c12}{rgb}{0,0.1,.9} %
    \begin{center}
     \resizebox{\textwidth}{!}{
     \begin{tikzpicture}
     \foreach  \x in {-12,0,12}{
     \begin{scope}[shift={(\x,-4)}]
 	 \node[draw,circle,fill=white] at (-4,-4) (A) {\phantom{$1$}};
	 \node[draw,circle,fill=white] at (0,-4) (B) {\phantom{$1$}};
 	 \node[draw,circle,fill=white] at (4,-4) (C) {\phantom{$1$}};
 	 \node[draw,circle,fill=white] at (-16/3,-8) (E) {\phantom{$1$}};
 	 \node[draw,circle,fill=green_] at (-8/3,-8) (F) {\phantom{$1$}};
 	 \node[draw,circle,fill=white] at (-4/3,-8) (G) {\phantom{$1$}};
 	 \node[draw,circle,fill=green_] at (4/3,-8) (H) {\phantom{$1$}};
 	 \node[draw,circle,fill=white] at (8/3,-8) (I) {\phantom{$1$}};
 	 \node[draw,circle,fill=green_] at (16/3,-8) (J) {\phantom{$1$}};
       \draw[color=c\x,very thick] (0,0)--(A)--(E);
     \draw[very thick,color=orange_](B)--(G);
     \draw[very thick,color=orange_](C)--(I);
     \draw[dashed] (H)--(B);
     \draw[dashed] (B)--(0,0)-- (C);
     \draw[dashed] (C)--(J);
     \draw[dashed] (A)--(F);
     \end{scope}
   \node[draw,circle,fill=white] at (0,0) (0) {\phantom{$1$}};
     \node[draw,circle,fill=white] at (0,-4) (b) {\phantom{$1$}};
     \node[draw,circle,fill=white] at (-12,-4) (a) {\phantom{$1$}};
     \node[draw,circle,fill=white] at (12,-4) (c) {\phantom{$1$}};
      \draw[color=c-12,very thick] (a)--(0);
       \draw[dashed](c)--(0)--(b);
   }
    \end{tikzpicture}}
     \caption{The smallest lush hedge of height $H=3$, having $\ell_1=9$, $\ell_2=6$, $\ell_3=2$ and $\ell_4=1$. The solid edges show a spanning subgraph $9{\color{green_!70!black}P_1}\cup 6{\color{orange_!70!black}P_2}\cup 2{\color{c0!70!black}P_3}\cup 1{\color{c-12!70!black}P_4}$ that could be used in the path-to-hedge constriction.%
     }
     \label{fig:T31}
    \end{center}
\end{figure}

The (unordered) multiplicity list \[{\ordm}=\{\ell_1+\ell_3,\ell_2+\ell_4,\ell_2,\ell_3,\ell_3,1,1,1\}\] is a critical multiplicity list for $T$, and our results show that it is realizable in $\S(T)$. 
Indeed, for $\Lambda \in \mathcal{B}$ there exists a matrix in $\S(T)$ cospectral with $C_4^\Lambda$, hence there exists a symmetric matrix $A\in \PHset(C_4^\Lambda,T)$. We know that $\spec(C_i^\Lambda)\cap \Lambda=\{\alpha_i,\beta_i\}$ for $i=2,3,4$, and $\spec(C_1^\Lambda)=\{\alpha_1\}$. Moreover, by Remark~\ref{rk:rn}, $\spec(C_3^\Lambda)=\{\alpha_1,\beta_3,\delta_1\}$ where $\delta_3$ is the root of the linear function $r_3(x)$; and $\spec(C_4^\Lambda)=\{\alpha_2,\beta_4,\delta_2,\delta_3\}$ where $\delta_2,\delta_3$ are the roots of the quadratic $r_4(x)$. It follows from Lemma~\ref{lem:steps} that $\alpha_1,\alpha_2,\beta_2,\beta_3,\beta_4,\delta_1,\delta_2,\delta_3$ are distinct.
Since\[\spec(A)=\bigcup_{i=1}^4\ell_i\sigma(C_i^\Lambda)=\{\alpha_1^{(\ell_1+\ell_3)},\alpha_2^{(\ell_2+\ell_4)},\beta_2^{(\ell_2)},\beta_3^{(\ell_3)},\delta_1^{(\ell_3)},\beta_4,\delta_2,\delta_3\}\]
it follows that the unordered multiplicity list of~$A$ is precisely $\ordm$.
Moreover, since ${\ordm}$ is critical, by Theorem~\ref{thm:main}, every matrix $A\in \S(T)$ with unordered multiplicity list $\ordm$ lies in $\PHset(C_4^\Lambda,T)$ as described above,  for some $\Lambda\in\mathcal B$. 

Now suppose that $A$ is any matrix in $\S(T)$ so that $A$ has unordered multiplicity list $\ordm$.  We will show that there is necessarily a labelling $\lambda_1,\dots,\lambda_8$ of the distinct eigenvalues of $\ordm$ so that the following constraints are satisfied:
\begin{gather*}
    \lambda_3+\lambda_6=\lambda_2+\lambda_7,\\
    \lambda_3 + \lambda_5 + \lambda_6= \lambda_1 + \lambda_4 + \lambda_8,\\
     (\lambda_2-\lambda_3)(\lambda_5-\lambda_3)(\lambda_7-\lambda_3) = (\lambda_1-\lambda_3)(\lambda_4-\lambda_3)(\lambda_8-\lambda_3).
\end{gather*}
Observe that taking a linear combination of the first two constraints gives
  $$\lambda_5+3\lambda_3+3\lambda_6=2\lambda_2+2\lambda_7+\lambda_1+\lambda_4+\lambda_8,$$
which was already observed by Ferrero et~al.~\cite[Theorem~4.3]{FFHHLMNS}. 
  
Firstly, since ${\ordm}$ is critical, by Theorem~\ref{thm:main}, every matrix $A\in \S(T)$ with unordered multiplicity list $\ordm$ lies in $\PHset(C_4^\Lambda,T)$ as described above,  for some $\Lambda\in\mathcal B$. 
We have $\tr(C_2^\Lambda)=\tr(C_3^\Lambda)-\alpha_1$, so
\[\alpha_2+\beta_2=\beta_3+\delta_1.\]
Similarly, $\tr(C_3^\Lambda)=\tr(C_4^\Lambda)-\alpha_2$, so
\[ \alpha_1+\beta_3+\delta_1=\beta_4+\delta_2+\delta_3.\]
Recall that $p_4(x)=(x-\alpha_2)p_3(x)-b_4p_2(x)$, so
\begin{align*}
    \frac{p_4(x)}{x-\alpha_2}=(x-\beta_4)(x-\delta_2)(x-\delta_3)
    =p_3(x)-b_4\frac{p_2(x)}{x-\alpha_2}.
\end{align*}
Since $p_2(\beta_2)=0$ and $p_3(x)=(x-\alpha_1)(x-\beta_3)(x-\delta_1)$, this yields the nonlinear identity
\begin{equation}\label{eq:nonlinear}
(\beta_2-\beta_4)(\beta_2-\delta_2)(\beta_2-\delta_3)=(\beta_2-\alpha_1)(\beta_2-\beta_3)(\beta_2-\delta_1).
\end{equation}
So we recover the constraints above by labelling the eigenvalues of~$A$ as $(\lambda_1,\lambda_2,\dots,\lambda_8)=(\beta_4,\beta_3,\beta_2,\delta_2,\alpha_1,\alpha_2,\delta_1,\delta_3)$.

The unordered multiplicity list ${\ordm}$ is realized by at least one ordered multiplicity list ${\ordm}^\prime$ of length $k = 8$, for which $\RS(T, \ordm^\prime)$ will be a non-empty subset of the interior of the $6$-dimensional standard simplex $\simplex_6$.
It follows from the above discussion, firstly that all of $\RS(T, \ordm^\prime)$ lies within
a $4$-dimensional affine subspace of $\simplex_6$, and secondly
that $\RS(T, \ordm^\prime)$ lives on a $3$-dimensional variety within that $4$-dimensional subspace
that is carved out by a cubic constraint. 
\end{example}

To explicitly illustrate that $\RS(T, \ordm)$ is not convex, we continue the example above by fixing some of the eigenvalues of matrices in $\PHset(C_4^\Lambda,T)$. 

\begin{example}\label{ex:non-convexity}
 Consider the matrix:
$$C_4^{\Lambda}=\left(
\begin{array}{cccc}
 \frac{28-30 x}{9(4-3 x) } & \frac{4x (1-x) }{3(4-3 x)} & 0 & 0 \\
 1 & x  & \frac{2 (8-9 x)}{27 (4-3 x)} & 0 \\
 0 & 1 & \frac{27 x^2-75 x+40}{9(4-3 x) } & \frac{(3 x-1) \left(27 x^2-66 x+28\right)}{27
   (4-3 x)} \\
 0 & 0 & 1 & x \\
\end{array}
\right)\in \cs(P_4)$$
for $\tfrac13<x<\frac{1}{9} \left(11-\sqrt{37}\right)$, where
$$\Lambda=(\alpha_1,\alpha_2,\beta_2,\beta_3,\beta_4)=\left(x,\frac{28-30 x}{9(4-3 x) },\frac{1}{3},\frac{1}{9},\frac{-27 x^2+24 x+4}{9(4-3 x) }\right) \in \mathcal{B}.$$
Here $\Lambda$ is chosen so that $\beta_2=\frac{1}{3}$,  $\beta_3=\frac{1}{9}$ and $C_4^{\Lambda}$ has eigenvalues $0$ and $1$, 
which implies that $\alpha_2=\frac{28-30 x}{9(4-3 x) }$ and $\beta_4=\frac{-27 x^2+24 x+4}{9(4-3 x) }$, i.e. $\delta_2=0$ and $\delta_3=1$. With this choice the order of the eigenvalues of  $A\in \PHset(C_4^{\Lambda},T)$ is as follows:
$$\delta_2<\beta_3<\beta_2<\beta_4<\alpha_1<\alpha_2<\delta_1<\delta_3.$$

Let $\Lambda_1$ be obtained from $\Lambda$ by setting $x=2/5$ and  $\Lambda_2$ be obtained from $\Lambda$ by setting $x=1/2$.  For $A\in \PHset(C_4^{\Lambda_1},T)$ we have
$$\sigma_1=\sigma(A)=\left\{0,\frac{1}{9}^{(\ell_3)},\frac{1}{3}^{(\ell_2)},\frac{116}{315},\frac{2}{5}^{(\ell_1+\ell_3)},\frac{40}{63}^{(\ell_2+\ell_4)},\frac{6}{7}^{(\ell_3)},1\right\},$$
and for $A\in \PHset(C_4^{\Lambda_2},T)$ we have
$$\sigma_2=\sigma(A)=\left\{0,\frac{1}{9}^{(\ell_3)},\frac{1}{3}^{(\ell_2)},\frac{37}{90},\frac{1}{2}^{(\ell_1+\ell_3)},\frac{26}{45}^{(\ell_2+\ell_4)},\frac{4}{5}^{(\ell_3)},1\right\}.$$
For the ordered multiplicity list ${\ordm_0}:=(1,\ell_3,\ell_2,\ell_1+\ell_3,\ell_2+\ell_4,\ell_3,1)$ we now have 
\begin{align*}
{\bf p}_1&=\left(\frac{1}{9},\frac{2}{9},\frac{11}{315},\frac{2}{63},\frac{74}{315},\frac{2}{9},\frac{1}{7}\right) \in \RS(T, \ordm_0), \\{\bf p}_2&= \left(\frac{1}{9},\frac{2}{9},\frac{7}{90},\frac{4}{45},\frac{7}{90},\frac{2}{9},\frac{1}{5}\right)\in \RS(T, \ordm_0).
\end{align*} 
It is straightforward to check that \eqref{eq:nonlinear} is satisfied for $\sigma_1$ and $\sigma_2$, but not for $t \sigma_1+(1-t)\sigma_2$ for any $t\in (0,1)$. Equivalently, $t{\bf p}_1+(1-t){\bf p}_2 \not\in\RS(T, \ordm_0)$ for all $t \in (0,1)$. 
\end{example}

\section{Resolution of two conjectures}\label{sec:S(T)}

Our results allow us to supply counterexamples to the conjectures listed below.
(While the original statements were formulated for $\mathcal{H}(T)$, the set of Hermitian matrices with graph $T$, rather than $\S(T)$, this is easily seen to be equivalent to the statements below.)
\begin{conjectures}\cite{MR3557827}\label{conj:split+01}
  Let $T$ be a tree and suppose that ${\bf m}=\{m_1,m_2,...,m_k\}$ is an unordered multiplicity list which is realizable by a matrix in $\S(T)$. 
  \begin{enumerate}
      \item\label{conj:split} (Splitting conjecture, \cite[Conjecture~6]{MR3557827}). For any $j\in [k]$ so that $m_j\geq 2$,  the unordered multiplicity list \[({\bf m}\setminus\{m_j\})\cup \{m_j-1,1\}=\{m_1,...,m_j-1,...,m_k,1\}\] is also realizable in $\S(T)$.
      \item\label{conj:01} (Zero-one conjecture, \cite[Conjecture~10]{MR3557827}). The unordered multiplicity list ${\ordm}$ is realizable by a matrix in $\S(T)$ with every off-diagonal entry equal to $0$ or $1$.
  \end{enumerate}
\end{conjectures}

Conjecture~\ref{conj:split+01}(\ref{conj:split}) has been proved for linear trees and trees on at most 12 vertices,~\cite{MR3557827,MR4357320}. Moreover, the bifurcation lemma of  Fallat, Hall, Lin and Shader~\cite{FALLAT202270} verifies that refinements of realizable multiplicity lists are also realizable under the additional assumption that the corresponding matrix has the Strong Multiplicity Property. We note that it is not possible to realize a \highMultiplicityList{} for a tree by a matrix with certain strong properties, see~\cite[Corollary~29]{MR3665573}, and the matrices constructed by our process do not have the Strong Multiplicity Property.  In the next example we show that any lush hedge $T$ of height $H\ge3$ is a counterexample to the splitting conjecture.

\begin{example}\label{counterexample1}
    (Counterexample to the splitting conjecture of~\cite{MR3557827}).
 Let $T$ be a lush hedge of height $H\geq 3$ and note that $\ell_3(T)\ge 2$ in this case. Let
  \[m_1=M(T^{(0)}),m_2=M(T^{(1)}),n_2=\widehat{M}(T^{(1)}),n_3=\widehat{M}(T^{(2)}), n_4=\widehat{M}(T^{(3)}).\] By Theorem~\ref{thm:main}, there is a  multiplicity list $\ordm$ so that $\{m_1,m_2,n_2,n_3,n_4\}\subseteq\ordm$ and $\ordm$ is realizable in $\S(T)$. Choose such an \[\ordm=\{m_1,m_2,n_2,n_3,n_4\}\cup\{m_6,m_7,\dots,m_k\}\] of maximal total length, $k$, realized by $A\in \S(T)$. By Theorem~\ref{thm:main}, we have $A\in \PHset(C_{H+1}^\Lambda,T)$ for some $\Lambda=(\alpha_1,\alpha_2,\beta_2,\beta_3,\beta_4)\in \mathcal B$. By %
  direct calculation,
  $C_3^\Lambda$ has eigenvalue $\lambda:=\alpha_2+\beta_2-\beta_3$ in addition to $\{\alpha_1,\beta_3\}$, and $\mult(\lambda,A)\ge \ell_3(T)\ge 2$. Moreover, $\lambda\not\in \{\alpha_1,\alpha_2,\beta_2,\beta_3,\beta_4\}$: since $\lambda\in \spec(C_3^\Lambda)\setminus\{\alpha_1,\beta_3\}$, we have $\lambda\not\in \{\alpha_1,\beta_3\}\cup \spec(C_2^\Lambda)\cup \spec(C_4^\Lambda)\supseteq\{\alpha_1,\alpha_2,\beta_2,\beta_3,\beta_4\}$, by Lemma~\ref{lem:steps}. Hence, we have $m_j=\mult(\lambda,A)\ge2$ for some $j$ with $6\le j\le k$. The list
  \[ \ordm'=(\ordm \setminus \{m_j\})\cup \{m_j-1,1\}\] has length $k+1$ and contains $\{m_1,m_2,m_3,m_4,m_5\}$, so it is not realizable in $\S(T)$, by our choice of $k$. Hence, the preceding conjecture does not hold for~$T$.
  
  For a specific counterexample, consider the height~$3$ lush hedge $T$ on $31$ vertices in Figure~\ref{fig:T31}. In the notation above, we have $(m_1,m_2,n_2,n_3,n_4)=(11,7,6,2,1)$. The unordered multiplicity list \[{\bf m}=\{11,7,6,2,2,1,1,1\}\] is realizable in $\S(T)$, but \[{\bf m'}=\{11,7,6,2,1,1,1,1,1\}\] is not.
\end{example}

\begin{example}\label{counterexample2}
(Counterexample to the zero-one conjecture of~\cite{MR3557827}).
 Let $T$ be any lush hedge of height~$2$ which has two vertices $x,y$ of height~$1$ so that $x$ and $y$ have different numbers of children. (The smallest example has~$11$ vertices and is obtained by appending one extra leaf to the Barioli-Fallat tree~$\bftree$, and is shown in Figure~\ref{fig:BFtree-plus-leaf}).
  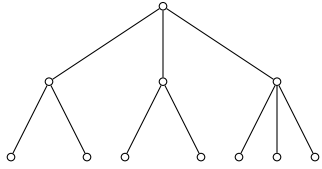
\begin{figure}[htb]
    \centering{%
    \begin{tikzpicture}[scale=0.5]
   \node[draw,circle,fill=white] at (0,0) (D) {};
    \node[draw,circle,fill=white] at (-3,-2) (A) {};
    \node[draw,circle,fill=white] at (-4,-4) (E) {};
    \node[draw,circle,fill=white] at (0,-2) (B) {};
    \node[draw,circle,fill=white] at (-1,-4) (G) {};
    \node[draw,circle,fill=white] at (3,-2) (C) {};
    \node[draw,circle,fill=white] at (2,-4) (I) {};
    \node[draw,circle,fill=white] at (-2,-4) (F) {};
    \node[draw,circle,fill=white] at (1,-4) (H) {};
    \node[draw,circle,fill=white] at (3,-4) (J) {};
    \node[draw,circle,fill=white] at (4,-4) (K) {};
        \draw(D)--(A)--(E);
        \draw(B)--(G); 
        \draw (K)--(C)--(I);
        \draw (H)--(B)--(D)--(C)--(J);
        \draw (A)--(F);
    \end{tikzpicture}}
    \caption{The smallest tree considered in Example~\ref{counterexample2}.}\label{fig:BFtree-plus-leaf}
\end{figure}
Writing $\ell_i=\ell_i(T)$, consider the multiplicity list $\ordm:=\{\ell_1+\ell_3,\ell_2,\ell_2,1,1\}$. It is the unique critical multiplicity list for $T$ and by %
Theorem~\ref{thm:section-4-highlights}, $\ordm$ is realizable in $\S(T)$. Suppose for a contradiction that $A\in \S(T)$ has every off-diagonal entry equal to~$1$, and $A$ has unordered multiplicity list $\ordm$. By Theorem~\ref{thm:main}, $A\in \PHset(C,T)$ for some $3\times 3$ matrix $C=(c_{i,j})\in \S(P_3)$, which implies that for any vertex $z$ in~$T$ with height~$1$, we have \[c_{2,3}^2=\sum_{z'\in \children(z)}a_{z,z'}^2=|\children(z)|.\]
 Since $x$ and $y$ have a different number of children, this is a contradiction.
\end{example}

\section{A spectacular failure of spectral arbitrariness}
 \label{sec:rigid}
 
 We have in Example \ref{ex:T31-M-construction} a case where the dimension of $\RS(T, \ordm)$ is merely~$3$, whereas the open simplex $\simplex_6$ in which it lies, and which it would completely fill if $\ordm$ were spectrally arbitrary for $T$, has dimension $6$.
 In fact, more generally, any time that $T$ is a lush hedge and $\ordm$ is a \highMultiplicityList{}
 that can be realized within $\S(T)$, Theorem \ref{thm:main} implies that every realization of $\ordm$ depends on only $5$ choices, or only $3$ choices up to shifting and scaling.
 It follows that in these cases $\RS(T, \ordm)$ has dimension at most $3$ (and, incidentally, an increasing number of non-linear constraints of increasing degree) no matter how large might be the dimension of the open simplex $\simplexk$ in which it lies.
 Considering that spectral arbitrariness is equivalent to $\RS(T, \ordm) = \simplexk$, this already represents, for any realizable \highMultiplicityList{}, a fairly spectacular failure of spectral arbitrariness.
 
But we can do even worse.
We now show that a matrix with carefully chosen eigenvalues may give rise to a multiplicity list for which the relative spacing of eigenvalues that can be realized in $\mathcal{S}(T)$ is \emph{completely} rigid.
\begin{theorem}
  \label{thm:rigid}
  Let $T$ be a lush hedge of height $H \ge 8$.
  \begin{enumerate}
      \item There exist both a matrix $A \in \mathcal{S}(T)$ and an ordered multiplicity list $\ordmrigid(T) = \ordm(A)$ such that $\ordmrigid(T)$ contains, in addition to the five multiplicities required for a \highMultiplicityList{}, multiplicities
      \begin{itemize}
          \item $m_{3, 7} \ge \ell_3 + \ell_7$,
          \item $m_{4, 8} \ge \ell_4 + \ell_8$, and
          \item $m_{4, 9} \ge \ell_4 + \ell_9$.
      \end{itemize}
      \item Given any matrix $B \in \mathcal{S}(T)$ whose ordered multiplicity list $\ordm(B)$ is a \highMultiplicityList{} and contains three additional multiplicities at least as large, respectively, as $\ell_3+\ell_7$, $\ell_4+\ell_8$, and $\ell_4+\ell_9$, either $\ordm(B)$ or $\ordm(-B)$ is identical to $\ordmrigid(T)$.
      \item Given any matrix $B \in \mathcal{S}(T)$ with
      $\ordm(B) = \ordmrigid(T)$, the spectrum of $B$ is identical to the spectrum of~$A$ up to shifting and scaling, i.e.,
    the moduli space $\RS(T, \ordmrigid(T))$ consists of a single point.
  \end{enumerate}
 \end{theorem}
 
\begin{proof}
If there is a matrix $A \in \mathcal{S}(T)$ whose ordered multiplicity list $\ordm(A)$ is a \highMultiplicityList{} which contains three additional multiplicities $m_{3,7}$, $m_{4,8}$, and $m_{4,9}$ at least as large as $\ell_3+\ell_7$, $\ell_4+\ell_8$, and $\ell_4+\ell_9$, respectively, then by Theorem \ref{thm:main} the spectrum of~$A$ is determined entirely by some choice $\Lambda\in \mathcal{B}$ of five special eigenvalues. It also follows that each individual multiplicity in $\ordm(A)$ is a sum of terms (possibly a single term) $\ell_i=\ell_i(T)$ for distinct $i$, and additionally that any particular term $\ell_i$ occurs as a summand in precisely $i$ of the multiplicities.

The five special multiplicities of a realizable \highMultiplicityList{} account for the single $\ell_1$ summand, for both of the $\ell_2$ summands, and for exactly two of the $i$ total $\ell_i$ summands for each $i \ge 3$.
 None of the remaining non-special multiplicities can include both an $\ell_3$ summand and an $\ell_4$ summand, by strict eigenvalue interlacing for paths. By Lemma \ref{lem:decreasing-ells}, it now follows that the largest non-special multiplicity (equal to $m_{3,7}$) must have an $\ell_3$ summand, and that the second-largest and third-largest non-special multiplicities must each have an $\ell_4$ summand. By the same lemma, each of them must also have at least one other summand $\ell_b$ no later in the decreasing list, respectively, than $\ell_7$, $\ell_8$, or $\ell_9$.
Hence, for the multiplicities $m_{3,7}$, $m_{4,8}$, and $m_{4,9}$, there must exist corresponding eigenvalues $\lambda_{3,7},\lambda_{4,8},\lambda_{4,9} \not\in \Lambda$ that appear at more than one level. 

  Recall the polynomials $r_k(x)$ for $k > 2$ defined in Remark \ref{rk:rn} which vanish at exactly those eigenvalues of level $k$ that are not a part of~$\Lambda$. 
The possibility of such a non-special eigenvalue appearing at two different levels $a$ and $b$ is precisely captured by the vanishing of the resultant $r_{a,b}$ of $r_a(x)$ and $r_b(x)$. Note that by strict interlacing $r_{a,a+1}$ cannot vanish.
By our reasoning above, to achieve the non-special multiplicity $m_{3,7}$ it is therefore necessary for $r_{3,b}$  to vanish for some $b\in \{5,6,7\}$, and then to achieve $m_{4,8}$ in addition, $r_{4,c}$ must vanish for some $c\in \{6,7,8\}$. %

 Let us call a linear combination of the variables in $\Lambda$ \emph{trivially nonzero} if it occurs in the factorization of an entry on the superdiagonal of $C_n^\Lambda$, for some $n$. In addition to $\binom{5}{2} = 10$ terms such as $(\alpha_1 - \alpha_2)$ that are nonzero precisely because the five eigenvalues $\Lambda$ are distinct, the term $(\alpha_2 + \beta_2 - \beta_3 - \beta_4)$ is also trivially nonzero, ensuring $\alpha_2+\beta_2\ne \beta_3+\beta_4$.  We adopt this terminology because no such expression can evaluate to zero when $\Lambda\in\mathcal{B}$.

  Direct calculation shows that each of the following resultants factors entirely into trivially nonzero terms:
  $r_{3,5}$, $r_{3,6}$, 
  $r_{4, 6}$, and $r_{4, 7}$.
 Each decomposition into trivially nonzero factors establishes that the corresponding eigenvalue level pairs cannot have any coincidences outside of the five special eigenvalues of~$\Lambda$.
 We conclude that $r_{3, 7}$ and $r_{4, 8}$ must both vanish in order, respectively, for the multiplicities $m_{3, 7}$ and $m_{4, 8}$ to be achieved.
 
 We would like to argue that similarly $r_{4, 9}$ must vanish in order for the multiplicity $m_{4, 9}$ to be achieved, but on the face of it
 this could also be achieved if $r_4(x)$ and $r_8(x)$ had two roots in common.
This possibility can be excluded by considering a generalization of the resultant that detects multiple shared roots between polynomials. This generalization can be obtained by taking the companion matrix of one polynomial and plugging it in to the other polynomial, as explained for example in \cite{Parker35}.
 For the specific case we need, let $R_4$ be the $2 \times 2$ companion matrix of the quadratic $r_4(x)$, specifically the companion matrix that places the negatives of the coefficients along the first row of the matrix.
 Suppose that the values of the five special eigenvalues $\Lambda$ are such that $r_4(x)$ and $r_8(x)$ share two roots; in that case $r_4(x)$, which is the characteristic polynomial of $R_4$, will be a factor of $r_8(x)$, and $r_8(R_4)$ will evaluate to the $2 \times 2$ zero matrix.
 We observe, however, that the $(2, 1)$ entry of $r_8(R_4)$ is given by
 \[\left(\beta _3-\beta _4\right) \left(\alpha _1-\beta _2\right) \left(\beta _2-\alpha _2\right) \left(\alpha _1-\beta _4\right) \left(\alpha _2+\beta _2-\beta _3-\beta _4\right),\]
 a product of trivially nonzero terms, and so $r_4(x)$ and $r_8(x)$ can have at most one root in common.

The above discussion implies that the three additional high multiplicities occur if and only if there is a choice of~$\Lambda$ within $\mathcal{B}$ which causes all three of the polynomials $r_{3, 7}$, $r_{4, 8}$, and $r_{4, 9}$ to vanish.
To simplify further calculations we use the freedom of shifting and scaling in order to make the assumption that $\Lambda$ includes setting $\beta_2$ to $-1$ and $\beta_4$ to $1$,
which in particular eliminates the denominator $\beta_4 - \beta_2$ so that all resultants are polynomials in the remaining variables $\alpha_1$, $\alpha_2$, and $\beta_3$.

 Let $r_{a,b}'$ denote the simplification of the resultant $r_{a,b}$ obtained by removing all trivially nonzero factors, and consider $r_{3, 7}^\prime$, $r_{4, 8}^\prime$, and $r_{4, 9}^\prime$. For example (under the assumption $\beta_2=-1$ and $\beta_4=1$), we have 
 \[
 r_{3,7}^\prime = \alpha_1\alpha_2 - \alpha_1\beta_3 - \beta_3 - \alpha_1 .
 \]
 Using a computer algebra system to solve $r_{3,7}'=r_{4,8}'=r_{4,9}'=0$ under the constraints that $\{\alpha_1,\alpha_2,\beta_3,-1,1\}$ are distinct and real and that $b_2=(1+\alpha_1)(\alpha_2-\alpha_1)$ is positive produces a unique solution for $\alpha_1,\alpha_2,\beta_3$. Moreover, each of these three values lies in 
 $\mathbb{Q}[\om]$, where $\om \approx 0.334981556$ is the smallest positive root of
\[
\om^{6} - 3 \, \om^{5} - 11 \, \om^{4} + 24 \, \om^{3} - 6 \, \om^{2} - 48 \, \om + 16 = 0.
\]
The exact values are as follows: $\beta_4=1$, $\beta_2=-1$,
 \[
\begin{array}{rcrcrcrcrcrcrcrl}
\alpha_1 \seq \dfrac{1}{90} \sleft \om^{5} \smin 4 \om^{4} \smin 16 \om^{3}
  \splus 49 \om^{2} \splus 44 \om \smin 74 \sright \sprox -0.604555194, &\\
&&&&&&&{}&&&&&& \\
\alpha_2 \seq \dfrac{1}{30} \sleft - 3 \om^{5} \splus 10 \om^{4} \splus 30 \om^{3}
  \smin 73 \om^{2} \splus 24 \om \splus 14 \sright \sprox 0.502965741, &\mspace{-12mu}\mbox{and} \\
&&&&&&&{}&&&&&& \\
\beta_3 \seq \dfrac{1}{60} \sleft - \om^{5} \splus \om^{4} \splus 25 \om^{3}
  \smin 22 \om^{2} \smin \mspace{-3mu}134 \om \splus 92 \sright \sprox 0.759864937,
  \end{array}
  \]
giving the coincident eigenvalues
\[
\begin{array}{rcrcrcrcrcrcrcrl}
\lambda_{3,7} \seq \dfrac{1}{60} \sleft - 5 \om^{5} \splus 19 \om^{4} \splus 35 \om^{3}
  \smin \mspace{-3mu}124 \om^{2} \splus \mspace{-3mu}182 \om \smin \mspace{-3mu}124 \sright \sprox -1.256899196, &\\
&&&&&&&{}&&&&&& \\
\lambda_{4,8} \seq \dfrac{1}{90} \sleft - 2 \om^{5} \smin \om^{4} \splus 41 \om^{3}
  \splus 37 \om^{2} \smin \mspace{-3mu}124 \om \smin 86 \sright \sprox -1.354063522,  &\mspace{-12mu}\mbox{and} \\
&&&&&&&{}&&&&&& \\
\lambda_{4,9} \seq \dfrac{1}{30} \sleft - 2 \om^{5} \splus 9 \om^{4} \splus 11 \om^{3}
  \smin 69 \om^{2} \splus 80 \om \smin 42 \sright \sprox -0.747525931.
\end{array}
\]
Having exhibited a unique (up to shifting and scaling) common root of $r_{3, 7}$, $r_{4, 8}$, and $r_{4, 9}$ that lies within $\mathcal{B}$, after first having established that this exact set of eigenvalue coincidences is necessary and sufficient for a \highMultiplicityList{} with the required three extra multiplicities, all three claims of the theorem follow.
\end{proof}

We have shown that for any lush hedge of height at least $8$, there is an ordered multiplicity list $\ordm$ for which $\RS(T, \ordm)$ consists of only a single point---an ordered multiplicity list for which spectral arbitrariness fails maximally.

For any lush hedge with a number of levels $H + 1$ up to $40$, the resulting placement of all eigenvalues, arranged by level, is illustrated by Figure~\ref{fig:40levels}.

\begin{figure}[htb]
    \centering
    \includegraphics[width=4.9in]{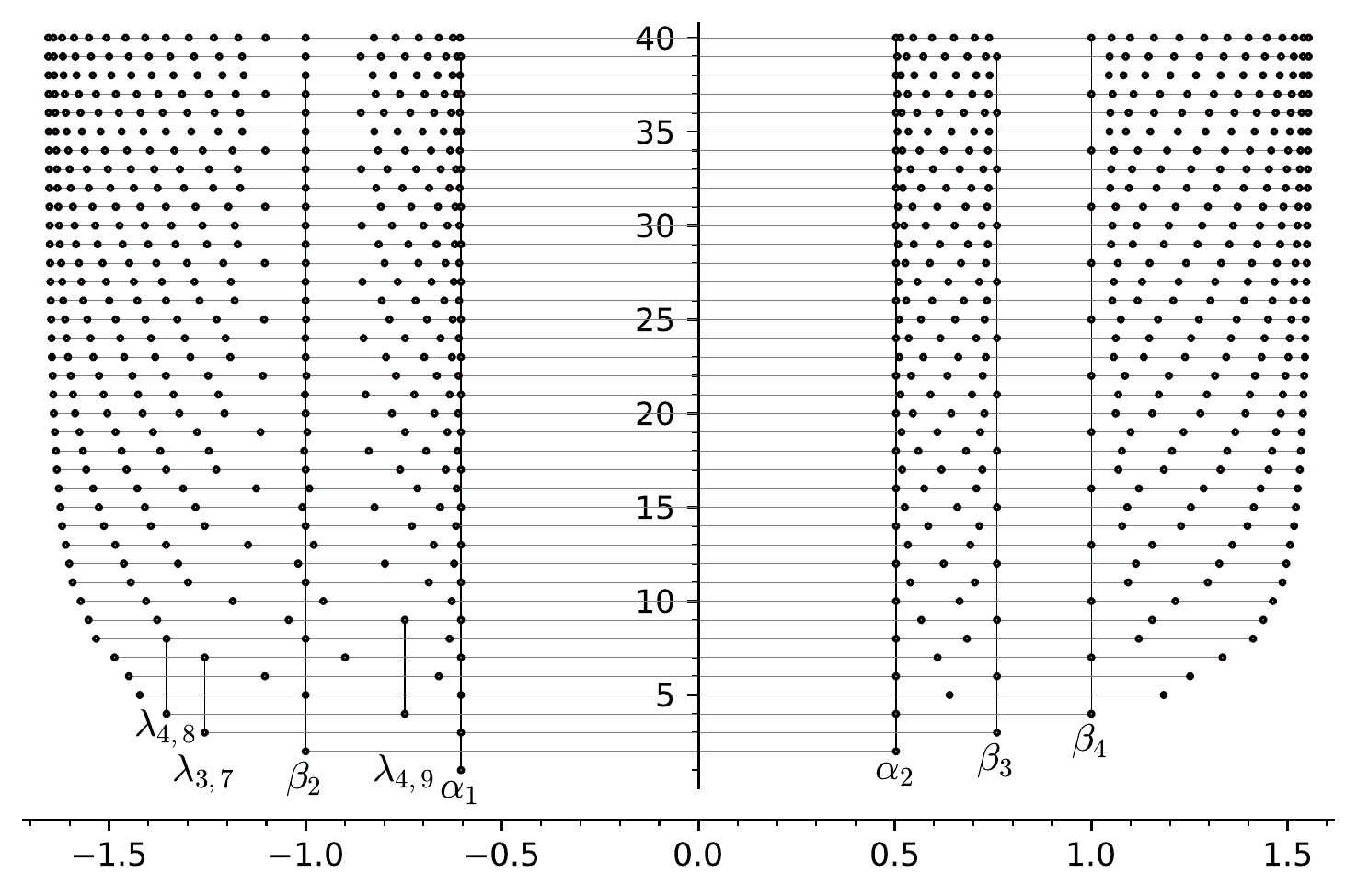}
    \caption{Depicted are $40$ levels of eigenvalues in the unique solution that arises in the proof of Theorem \ref{thm:rigid}.
    }
    \label{fig:40levels}
\end{figure}

\begin{remark}
Each of the three constraints $r_{3, 7}=0$, $r_{4, 8} = 0$, and $r_{4, 9} = 0$ reduces the dimension of $\RS(T, \ordm)$ by at most one, and together they reduce its dimension from $3$ down to $0$. It follows that a moduli space $\RS(T, \ordm)$ of dimension $2$ or of dimension $1$ can also be achieved by letting $\ordm$ be the ordered multiplicity list of a matrix obtained for some choice of~$\Lambda$ that satisfies only one or only two of these constraints.
\end{remark}

\bigskip

The next example applies Theorem \ref{thm:rigid} to a particular tree and a particular multiplicity list.

\begin{example}
 \label{ex:rigid}
 The smallest graph to which Theorem \ref{thm:rigid} may be applied is the smallest lush hedge of height $H = 8$, which we will call $\Texample$.
 This is nearly a complete ternary tree (root node of degree $3$, whose descendants have three children each and thus degree $4$), except that the leaf nodes come in pairs rather than in triples.
 The size of a lush hedge is at least exponential in its height, and this smallest example for height $8$ has $1 + 3 + 9 + 27 + 81 + 243 + 729 + 2187 + 4374 = 7654$ vertices in total.
 Taking differences between the number of vertices in consecutive levels, we obtain
$\ell_1 = 2187$,
$\ell_2 = 1458$,
$\ell_3 = 486$,
$\ell_4 = 162$,
$\ell_5 = 54$,
$\ell_6 = 18$,
$\ell_7 = 6$,
$\ell_8 = 2$, and
$\ell_9 = 1$.

Substituting in, for $\Lambda$, the algebraic numbers supplied in the proof of Theorem \ref{thm:rigid}, we obtain a matrix $C^\Lambda_9$ for which the path-to-hedge construction on $\Texample$ gives $\PHset(C^\Lambda_9,T_8)$, a parameterized family of very large matrices whose shared spectrum we know without having to write out any one of these matrices.
Calculating the eigenvalues of the appropriate submatrices of $C_9^\Lambda$ (or
referring to the lower $9$ levels of Figure \ref{fig:40levels}) we observe that there are no additional eigenvalue coincidences between levels other than those that are specifically required, and we deduce that the unordered multiplicity list of~$A$ will take the form
\[
\ordm(A) \mspace{5mu} =
\begin{array}{ll}
\mspace{-6mu}(
\ell_1 + \ell_3 + \ell_5 + \ell_7 + \ell_9 = 2734,
&(\mbox{for }\alpha_1) \\
\ell_2 + \ell_4 + \ell_6 + \ell_8 = 1640,
&(\mbox{for }\alpha_2) \\
\ell_2 + \ell_5 + \ell_8 = 1514,
&(\mbox{for }\beta_2 = -1) \\
\ell_3 + \ell_6 + \ell_9 = 505,
&(\mbox{for }\beta_3) \\
\ell_3 + \ell_7 = 492, 
&(\mbox{for }\lambda_{3, 7}) \\
\ell_4 + \ell_7 = 168, 
&(\mbox{for }\beta_4 = 1) \\
\ell_4 + \ell_8 = 164,
&(\mbox{for }\lambda_{4, 8}) \\
\ell_4 + \ell_9 = 163, 
&(\mbox{for }\lambda_{4, 9}) \\
\ell_5 = 54, \ell_5 = 54, \ell_5 = 54, &\\
\ell_6 = 18, \ell_6 = 18, \ell_6 = 18, \ell_6 = 18, &\\
\ell_7 = 6, \ell_7 = 6, \ell_7 = 6, \ell_7 = 6, &\\
\ell_8 = 2, \ell_8 = 2, \ell_8 = 2, \ell_8 = 2, \ell_8 = 2, &\\
\ell_9 = 1, \ell_9 = 1, \ell_9 = 1, \ell_9 = 1, \ell_9 = 1, \ell_9 = 1)
\mspace{-80mu}&
\end{array}
\]
with ordered multiplicity list
\begin{multline*}
\ordmrigid(\Texample)=
(1, 2, 6, 18, 54, 1, 164, 492, 18, 1, 1514, 6, 163, 18, 2, 2734, \\
1640, 1, 6, 54, 2, 505, 168, 2, 1, 54, 18, 6, 2, 1).
\end{multline*}
As a check on the total number of eigenvalues, each value $\ell_i$ contributes to the unordered multiplicity list in exactly $i$ places, and the sum of the ordered multiplicity list is $7654$.
The eigenvalues of~$A$ corresponding to the eight largest multiplicities are exactly the values in $\mathbb{Q}[\om]$ specified in the proof of Theorem \ref{thm:rigid}.
The eigenvalues of~$A$ corresponding to the leftover repeated multiplicities, from $\ell_5$ through $\ell_9$, are the
roots of specific monic polynomials,  factors of $r_5(x)$ through $r_9(x)$, each with degree the number of times the multiplicity is repeated and each of which has coefficients in (and, not surprisingly, turns out to be irreducible over)
the field $\mathbb{Q}[\om]$.

By Theorem \ref{thm:rigid}, any matrix $B \in \mathcal{S}(\Texample)$
that achieves this particular ordered multiplicity list $\ordmrigid(\Texample)$, and that does so with an eigenvalue $-1$ of multiplicity $1514$ and an eigenvalue $1$ of multiplicity $168$, will have precisely the same spectrum as $A$.
If $\ordmrigid(\Texample)$ were spectrally arbitrary for $\Texample$, then
$\RS(\Texample, \ordmrigid(\Texample))$ would fill all of $\simplex_{28}$;
instead it is $28$ dimensions smaller and consists of a single point.

\end{example}

\section*{Acknowledgements}
The authors are grateful to the organizers of the American Institute of Mathematics Research Community on the \emph{Inverse eigenvalue problem for graphs}. Shaun Fallat's research is in part supported by an NSERC Discovery Grant, RGPIN--2019--03934. Polona Oblak received funding from Slovenian Research Agency (research core funding no.~P1-0222 and project no.~J1-3004).

\goodbreak 

\bibliographystyle{plain}
\bibliography{bibliography}
\end{document}